\documentclass[final,leqno]{siamltex704}
\usepackage{hyperref}
\usepackage{amsmath}
\usepackage{graphicx}
\usepackage{array}
\usepackage{lineno}
\usepackage[notcite,notref]{showkeys}
\usepackage{tikz}
\usepackage{amsfonts,amssymb}
\usepackage{dsfont}
\usepackage{pifont}
\usepackage{subfigure}

\newtheorem{acknowledge}{Acknowledge}
\renewcommand{\ldots}{\dotsc}
\newtheorem{model-problem}{Problem}

\newcommand{\bx}{\textbf{x}}
\newcommand{\by}{\textbf{y}}

\newcommand{\bz}{\textbf{z}}
\newcommand{\br}{\textbf{r}}

\newcommand{\bq}{\textbf{q}}

\def\T{{\mathcal T}}
\def\S{{\mathcal S}}

\def\E{{\mathcal E}}

\def\pT{{\partial T}}

\def\bn{{\bf n}}
\def\bq{{\bf q}}

\def\3bar{{|\hspace{-.02in}|\hspace{-.02in}|}}

\def\OO{{\cal O}}
\def\bq{\begin{equation}}
\def\eq{\end{equation}}
\def\bbq{\begin{equation*}}
\def\eeq{\end{equation*}}
\def\br{\begin{eqnarray}}
\def\er{\end{eqnarray}}
\def\brr{\begin{eqnarray*}}
\def\err{\end{eqnarray*}}

\def\O{\Omega}

\def\E{{\mathcal E}}

\def\pa{\partial}

\def\bn{{\bf n}}

\def\3bar{{|\hspace{-.02in}|\hspace{-.02in}|}}

\newtheorem{WEAK GALERKIN ALGORITHM}{WEAK GALERKIN ALGORITHM}
\newtheorem{SIMPLIFIED WEAK GALERKIN ALGORITHM}{SIMPLIFIED WEAK GALERKIN ALGORITHM}

\begin{document}

\setlength{\parindent}{0.25in} \setlength{\parskip}{0.08in}

\title{Superconvergence of Numerical Gradient for Weak Galerkin Finite Element Methods on Nonuniform Cartesian Partitions in Three Dimensions}
\author{Dan Li\thanks{Research Center for Computational Science, Northwestern Polytechnical University, Xi'an, Shannxi 710072, China. The research of Dan Li was supported in part by National Natural Science Foundation of China grant number 11471262.} \and
Yufeng Nie\thanks{Research Center for Computational Science, Northwestern Polytechnical University, Xi'an, Shannxi, 710129, China(yfnie@nwpu.edu.cn). The research of Yufeng Nie
was supported by National Natural Science Foundation of China
grants 11471262.}
\and
Chunmei Wang\thanks{Department of Mathematics, Texas Tech
University, Lubbock, TX  79409, USA. The research of Chunmei Wang was partially supported by National Science Foundation Awards DMS-1849483 and DMS-1648171. Email: chunmei.wang@ttu.edu.}}

\maketitle
\begin{abstract}
A superconvergence error estimate for the gradient approximation of the second order elliptic problem in three dimensions is analyzed by using weak Galerkin finite element scheme on the uniform and non-uniform cubic partitions. Due to the loss of the symmetric property from two dimensions to three dimensions, this superconvergence result in three dimensions is not a trivial extension of the recent superconvergence result in two dimensions \cite{sup_LWW2018} from rectangular partitions to cubic partitions. The error estimate for the numerical gradient in the $L^{2}$-norm arrives at a superconvergence order of ${\cal O}(h^r) (1.5 \leq r\leq 2)$ when the lowest order weak Galerkin finite elements consisting of piecewise linear polynomials in the interior of the elements and piecewise constants on the faces of the elements are employed. A series of numerical experiments are illustrated to confirm the established superconvergence theory in three dimensions.
\end{abstract}



\section{Introduction}
Superconvergence is a phenomenon where the numerical solution converges to the exact solution at a rate faster than generally expected. Superconvergence has been widely used in post-processing techniques to yield a more accurate approximation \cite{REELJW-2002}. Superconvergence has also been employed by the mesh refinement and adaptivity \cite{WCZ-2014,sup_CMAME1992} to yield a posterior error estimator \cite{PEE_MAJT2000,HCJW-2003,FLLWZ-2006,supposter_PIJNM11992,supposter_PIJNM21992}. There has been a variety of research work in superconvergence based on finite difference methods \cite{JAFRDG-2006,REXFZC-2018}, finite element methods \cite{ HFS_AMC2001, CBPE_MMA1996, HCJW-2003,CMCYQH1995,wangsup,sup_QDZQL1989,MZ_MMA2002}, discontinuous Galerkin methods \cite{SDG_SJAM2015}, hybridized discontinuous Galerkin methods \cite{CSA_SIAM2016}, smoothed finite element methods \cite{sup_CRC2016}, and weak Galerkin finite element methods \cite{AHSH2014,sup_LWW2018,sup_LW2018, LJW2018,SAPP_JSC2018,ellip_WY2013}.

In this paper, we are concerned with new developments of superconvergence of  weak Galerkin finite element method for second order elliptic boundary value problem (BVP) in three dimensions. To this end, we consider the second order elliptic problem in three dimensions: Find $u=u(x, y, z)$ satisfying
\begin{equation}\label{model}
\begin{split}
-\nabla \cdot ({\color{black}A}\nabla u) =& f,  \quad \mbox{in}~~~ \O, \\
u=&g,\quad \mbox{on}~~ \pa\O ,
\end{split}
\end{equation}
where $\O$ is an open bounded domain in $\mathbb{R}^3$ with Lipschitz continuous boundary $\pa \O$; $f=f(x, y, z)\in H^{-1}(\Omega)$ and $g=g(x, y, z)\in H^{\frac12}(\pa\O)$ are given functions defined on $\O$ and the boundary $\pa \O$, respectively. We assume that the diffusive coefficient tensor ${\color{black}A}=\{a_{ij}\}_{3\times3}$ is uniformly bounded, symmetric, and positive definite in $\O$.

The weak formulation of the second order elliptic model problem \eqref{model} using the usual integration by parts is as follows: Find $u\in H^1(\O)$  satisfying $u=g$ on $\pa\O$, such that
\begin{equation}\label{weakform}
({\color{black}A}\nabla u, \nabla v)=(f, v), \qquad  \forall v\in V,
\end{equation}
where $V=\{v\in H^1(\Omega): v=0 \ \text{on}\ \pa \O\}$.

Superconvergence for the gradient of the finite element approximation for the second order elliptic boundary value problem has been an active research topic for many years \cite{CBPE_MMA1996,HTJW-2002,MKPN1984,VKPDL1986,AHSIHLB-1996,sup_GFEM2006}. There have been various numerical methods for solving the second order elliptic equations \eqref{model}, such as finite element methods, finite volume methods, and finite difference methods etc. We shall focus on a newly-developed numerical method named ``weak Galerkin finite element method (WG-FEM)''  which is a natural extension of the classical Galerkin finite element methods. WG-FEM has several advantages over the classical Galerkin finite element methods: (1) WG-FEM is flexible to use discontinuous functions with interior information and boundary information; (2) WG-FEM is flexible to use polygons in two dimensions or polyhedra in three dimensions in mesh generation; (3) WG-FEM is stable and preserves the physical properties. WG-FEM has been widely applied to solve various partial differential equations such as elliptic interface problem \cite{ANWG_EIP2016}, Maxwell's equations \cite{WWmaxwell, WG_ME2017,MWYS_MAX2015}, the Helmholtz equation \cite{MWYS_HE2014}, wave equation \cite{DWGWE_JSC2017}, Stokes equations \cite{sup_LW2018,JWYSE-2016},  the div-curl system \cite{WWdivcurl}, the biharmonic problem \cite{WWhybird, WWbiharmonic}, the Cahn-Hilliard equation\cite{WZZZ_MC2018}, the singularly perturbed convection-diffusion-reaction problems \cite{LYZZ-WG-2018} etc.  Recently, the primal-dual weak Galerkin finite element method has been successfully developed to solve challenging problems such as the second order elliptic equation in non-divergence form \cite{PDWG_MC2018}, the Fokker-Planck equation \cite{fp2018} and the elliptic cauchy problems \cite{ECP2018,CW2018}.

Some superconvergence results were observed in the numerical experiments of  WG-FEM method on uniform meshes  for the gradient approximation  for the  elliptic equation in three dimensions (see $\|\nabla_{d}e_{h}\|$ in Table 4.11 \cite{ellip_NWA2013}) and the wave equation  (see $\|\nabla_{w}(e_{h})\|$ in Table II \cite{DWGWE_JSC2017}). The numerical results in
\cite{PDWG_MC2018} showed the superconvergence rate ${\cal{O}}(h^4)$ in the discrete $L^{2}$-norm on uniform triangular partitions. Recently, the superconvergence theory based on WG scheme has been developed and analyzed on non-uniform rectangular partitions for the second order elliptic problem \cite{sup_LWW2018} and stokes equation \cite{sup_LW2018}, respectively. In \cite{ellip_WY2013}, a superconvergence in $L^{2}$-norm was proved between the $L^{2}$ projection of the exact solution and its numerical approximation. In \cite{AHSH2014}, the authors studied the $H^{1}$- superconvergence of the WG-FEM method by $L^{2}$ projections introduced in \cite{ellip_JW2000}, and derived a superconvergence rate ${\cal{O}}(h^{1.5})$ or better by using the lowest order weak Galerkin element approximations for the elliptic problem.

There are some superconvergence results in three-dimensions in the literature \cite{HFS_AMC2001,REXFZC-2018, AHSKMK2010,VKPDL1986,k2005, RLZZ1008, ZZ1998}. The difficulty in the analysis of superconvergence for problems in three-dimensions lies in the loss of orthogonality and/or symmetry compared with the analysis for problems in two dimensions. In this paper, we shall extend the superconvergence result in \cite{sup_LWW2018} for the second order elliptic problem \eqref{model} from two dimensions to three dimensions. This is a non-trivial extension of \cite{sup_LWW2018} in both the analysis and numerical experiments. The main difficulty in this paper compared with \cite{sup_LWW2018} lies in that the symmetric property for the rectangular partitions in two dimensions is not available for the cubic partition in three dimensions. The innovative contribution in this paper is to develop the superconvergence order ${\cal{O}}(h^{r}) (1.5 \leq r\leq 2)$ for the numerical gradient for the second order problem in three dimensions.

The rest of this paper is organized as follows. In Section \ref{Section:WeakGradient}, we simply review the weak gradient operator as well as its discrete version. Section \ref{Section:WG-Scheme} is devoted to reviewing the WG-FEM finite element scheme for the second order elliptic problem \eqref{model} in three dimensions. A simplified WG-FEM scheme is derived in Section \ref{Section:WG-SWG}. The error equation for the simplified WG scheme is developed in Section \ref{Section:error equation}. In Section \ref{Section:TE:22:45}, some technical results are provided which are useful in the analysis of the superconvergence of WG method. Superconvergence theory is established in Section \ref{Section:SWG}.  In Section \ref{Section:NE}, a variety of numerical experiments are demonstrated to verify the established superconvergence theory.


\section{Weak Gradient and Discrete Weak Gradient}\label{Section:WeakGradient}
The classical gradient operator is the differential operator used in the weak formulation (\ref{weakform}) of the second order elliptic model problem (\ref{model}). In this section, we will briefly review the weak gradient operator as well as its discrete version which were first introduced in \cite{ellip_WY2013, ellip_MC2014}.

Let $T$ be any polyhedral domain with boundary $\partial T$. Denote by $v=\{v_{0},v_{b}\}$ a weak function on $T$, where the first and second components $v_{0}$ and $v_{b}$ represent the information of $v$ in the interior and on the boundary of $T$, respectively. Note that $v_{b}$ may not necessarily be related to the trace of $v_{0}$ on the boundary $\partial T$.  However, it is feasible to take $v_b$ as the trace of $v_{0}$ on $\partial T$.

We introduce the space of the weak functions on $T$, denoted by $W(T)$; i.e.,
\begin{equation*}
W(T)=\{v=\{v_0,v_b\}: v_0 \in L^2(T), v_b \in L^2(\partial T)\}.
\end{equation*}

The weak gradient of $v\in W(T)$, denoted by $\nabla_{w}v$, is defined as a linear functional in the dual space of $[H^1(T)]^3$ satisfying
\begin{equation}\label{weak gradient notation}
\langle\nabla_{w}v,\boldsymbol{\psi}\rangle_{T}=-(v_{0},\nabla\cdot\boldsymbol{\psi} )_{T}+\langle v_{b},\boldsymbol{\psi}\cdot \mathbf{n} \rangle_{\partial T},\quad \forall\boldsymbol{\psi}\in [H^1(T)]^3,
\end{equation}
where $\bn$ is the unit outward normal direction to $\partial T$.

Denote by $P_r(T)$ the set of polynomials on $T$ with total degree no more than $r$. A discrete version of $\nabla_{w} v$ for any $v\in W(T)$, denoted by $\nabla_{w, r, T} v$, is defined as the unique vector-valued polynomial in $[P_r(T) ]^3$ satisfying
\begin{equation}\label{disgradient}
(\nabla_{w, r, T} v, \boldsymbol{\psi})_T=-(v_0,\nabla \cdot \boldsymbol{\psi})_T+\langle v_b, \boldsymbol{\psi} \cdot  \textbf{n}\rangle_{\partial T},  \quad\forall\boldsymbol{\psi}\in [P_r(T)]^3.
\end{equation}

\section{Weak Galerkin Finite Element Scheme}\label{Section:WG-Scheme}
Let ${\cal T}_{h}$ be a polyhedral partition of the domain $\Omega\subset \mathbb R^3$ which is shape regular as specified in $\cite{ellip_MC2014}$. Denote by $\mathcal{E}_h$ the set of all flat faces in ${\cal T}_{h}$, and $\mathcal{E}_h^{0}=\mathcal{E}_h\setminus \partial\Omega $ the set of all interior flat faces. Denote by $h_{T}$ the size of the element $T\in {\cal T}_{h}$ and $h=\max_{T\in {\cal T}_{h}}h_{T}$ the mesh size of the partition ${\cal T}_{h}$.

Let $k\geq1$ be a given integer. We introduce the local discrete weak finite element space on each element $T\in {\cal T}_h$, denoted by $V(T, k)$; i.e.,
$$V(T, k)=\{v=\{v_{0},v_{b}\},v_{0} \in P_{k}(T),v_{b}  \in P_{k-1}(F), \ F\subset \partial T \}.$$
Patching $V(T, k)$ over all the elements $T\in {\cal T}_h$ through a common value $v_{b}$ on the interior interface $\E_h^0$ gives rise to a global weak finite element space $V_h$; i.e.,
$$
V_h=\{\{v_0,v_b\}: \{v_0,v_b\}|_T\in V(T, k),\ \mbox{$v_b$ is single-valued
on $\E_h$}\}.
$$
We further introduce the subspace of $V_{h}$ with vanishing  boundary values, denoted by $V_{h}^{0}$; i.e.,
$$V_{h}^{0}=\{\{v_{0},v_{b}\}\in V_{h}, v_{b}|_{F}=0, F\subset \partial \Omega \}.$$

For any $v\in V_h$, denote by $\nabla_{w}v$  the discrete weak gradient $\nabla _{w,k-1,T}v$ computed by using (\ref{disgradient})  on each element $T$; i.e.,
$$
(\nabla _{w,k-1}v)|_T= \nabla _{w,k-1,T}(v|_T).
$$
For simplicity of notation and without confusion,  we shall use $\nabla_{d} $ to denote $\nabla _{w,k-1}$; i.e.,
$$
\nabla _{d} v = \nabla _{w,k-1} v, \qquad \forall v\in V_h.
$$

For any $u=\{u_{0},u_{b}\}$ and $v=\{v_{0},v_{b}\}$ in $V_{h}$, we introduce the following two bilinear forms; i.e.,
\begin{eqnarray*}
 (A\nabla_{d}u,\nabla_{d}v)_h&=&\sum_{T\in\mathcal{T}_h}(A\nabla_{d}u,\nabla_{d}v)_{T},\\
s(u,v)&=&\rho h^{-1} \sum_{T\in\mathcal{T}_h}\langle Q_{b}u_{0}-u_{b},Q_{b}v_{0}-v_{b}\rangle_{\partial T},
\end{eqnarray*}
where $\rho>0$ is a parameter, and $Q_b$ is the usual $L^{2}$ projection operator from $L^2(F)$ onto $P_{k-1}(F)$.

We are in a position to review the weak Galerkin finite element method for the second order elliptic model problem (\ref{model}) based on the weak formulation (\ref{weakform}) \cite{ellip_WY2013, sup_LWW2018}.

\begin{WEAK GALERKIN ALGORITHM}\label{WGFEM} Find $u_{h}=\{u_{0},u_{b}\}\in V_{h}$ satisfying $u_b=\widetilde{Q}_b g$ on $\partial \Omega$ such that
\begin{equation}\label{WG-scheme}
({\color{black}A}\nabla_{d}u_{h},\nabla_{d}v_{h})_h+s(u_{h},v_{h})=(f,v_{0}),  \qquad \forall v_{h}\in V_{h}^{0},
\end{equation}
where $\widetilde{Q}_b g$ is a suitably-chosen projection operator of the Dirichlet boundary data $g$ onto the space of polynomials of degree $k-1$.
\end{WEAK GALERKIN ALGORITHM}

The approximate boundary data $\widetilde{Q}_b g$ may be chosen as
\begin{equation}\label{boundary-approximation}
\widetilde{Q}_b g:=Q_b g + \varepsilon_b,
\end{equation}
where $\varepsilon_b$ is a small perturbation of the $L^2$ projection $Q_b g$. A special example of the perturbation term is given by $\varepsilon_b=0$ such that $\widetilde{Q}_b g=Q_b g$. However, a non-zero perturbation $\varepsilon_b$ is necessary in the analysis of the superconvergence of the weak gradient approximation.

Note that the coefficient matrix of (\ref{WG-scheme}) is symmetric and positive definite for any $\rho>0$. Thus, the system (\ref{WG-scheme}) is solvable.

\section{Simplified Weak Galerkin Algorithm}\label{Section:WG-SWG}
In what follows of this paper, we shall focus on the lowest order of WG finite element, i.e., $k=1$. More precisely, the WG finite element $u_h$ is a piecewise linear polynomial in the interior and a piecewise constant on the boundary. The discrete weak gradient $\nabla_d u_h$ is a piecewise vector-valued constant.

A weak function $v =\{v_0, v_b\} \in V_{h}$ can be rewritten as
$$
v=\{v_{0},0\}+\{0,v_{b}\},
$$
which, for simplicity of notation and without confusion, will be denoted by $v=v_{0}+v_{b}$. Denote by $V_{0}=\{v_{0}=\{v_{0},0\}\in V_{h}\}$  the interior space, and $V_{b}=\{v_{b}=\{0,v_{b}\}\in V_{h}\}$  the boundary space, respectively. It is easy to check that $\nabla_{d}v_{0}=0$ from the definition of discrete weak gradient \eqref{disgradient}. Thus, the weak Galerkin algorithm \eqref{WG-scheme} can be simplified as follows: Find $u_{h}=\{u_{0},u_{b}\}\in V_{h}$ satisfying $u_{b}=\widetilde{Q}_b g$ on $\pa \O$ such that
\begin{equation}\label{a6}
({\color{black}A}\nabla_{d}u_{b},\nabla_{d}v_{b})_{h}+s(u_{h},v_{h})=(f,v_{0}), \qquad \forall v_{h}\in V_{h}^{0}.
\end{equation}

We introduce an extension operator ${\S}$ mapping $v_{b}\in P_{0}(\partial T)$ to a function in $P_1 (T)$ such that
\begin{equation}\label{a7}
\langle{\S}(v_{b}),Q_{b}\psi\rangle_{\partial T}\,=\,\langle v_{b},\psi\rangle_{\partial T},
\quad \forall\psi\in P_{1}(T).
\end{equation}
This implies
\begin{equation}\label{L2}
\begin{split}
 \langle Q_{b}u_{0}-u_{b},Q_{b}{\S}(v_{b})-v_{b}\rangle_{\partial T} =&
  \langle -u_{b},Q_{b}{\S}(v_{b})-v_{b}\rangle_{\partial T}\\
 =&\langle Q_{b}{\S}(u_{b})-u_{b},Q_{b}{\S}(v_{b})-v_{b}\rangle_{\partial T}.
\end{split}
\end{equation}

Letting $v_{h}=\{{\S}(v_{b}),v_{b}\}\in V_{h}^{0}$ in \eqref{a6}, and using (\ref{L2}), we obtain a simplified weak Galerkin finite element scheme.

\begin{SIMPLIFIED WEAK GALERKIN ALGORITHM} Find $u_{b}\in V_{b}^{g}$ satisfying
\begin{equation}\label{SWG}
({\color{black}A}\nabla_{d}u_{b},\nabla_{d}v_{b})_{h}+\rho h^{-1} \sum_{T\in\mathcal{T}_h}\langle Q_{b}{\S} (u_{b})-u_{b},Q_{b}{\S} (v_{b})-v_{b}\rangle_{\partial T}=(f,\S (v_{b})),
\end{equation}
for any $v_b\in V_b^0$. Here, $V_b^0=\{v_b\in V_b: v_b|_{\partial \Omega}=0\}$, and $V_b^g=\{v_b\in V_b: v_b|_{\partial \Omega}=\widetilde{Q}_b g\}$.
\end{SIMPLIFIED WEAK GALERKIN ALGORITHM}

\section{Error Equations}\label{Section:error equation}
In this section, we will derive an error equation for the simplified weak Galerkin finite element algorithm \eqref{SWG}, which will play an important role in the analysis of the superconvergence error estimates in Section \ref{Section:SWG}. For the convenience of analysis, we assume the coefficient tensor $A$ in the model problem \eqref{model} is a piecewise matrix-valued constant with respect to the finite element partition ${\cal T}_h$. However, the results can be generalized to the variable coefficient tensor $A$ without any difficulty, provided that the coefficient tensor $A$ is piecewise smooth.

On each element $T\in {\cal T}_h$, denote by $Q_{0}$ and $Q_{b}$ the usual $L^{2}$ projection operators onto $P_{1}(T)$ and $P_{0}(F)$, respectively. Denote by $\mathbb{Q}_{h}$ the usual $L^{2}$ projection operator onto $[P_{0}(T)]^{3}$. The $L^{2}$ projection operators $Q_{b}$ and $\mathbb{Q}_{h}$ satisfy the commutative property \cite{ellip_WY2013, sup_LWW2018}:
\begin{eqnarray}\label{COMM-PRO}
\nabla_{d}Q_{b}w=\mathbb{Q}_h\nabla w,\qquad \forall w\in H^{1}(T).
\end{eqnarray}

Denote by $e_{b}=Q_{b}u-u_{b}$ the error function between the WG solution and the $L^2$ projection of the exact solution of the model problem (\ref{model}). For the convenience of analysis, we introduce the flux variable $\textbf{q}=A\nabla u$.

\begin{lemma}\label{lemma equation}
The error function $e_b$ satisfies the following \textit{error equation}
\begin{equation}\label{a9}
({\color{black}A}\nabla_{d}e_{b},\nabla_{d}v_{b})_{h}+\rho h^{-1} \sum_{T\in\mathcal{T}_h}\langle Q_{b}{\S}(e_{b})-e_{b},Q_{b}{\S}(v_{b})-v_{b}\rangle_{\partial T}=\zeta_{u}(v_{b}),
\end{equation}
for any $v_b\in V_b^0$, where
\begin{equation}\label{a10}
\begin{split}
\zeta_{u}(v_{b})=&\sum_{T\in\mathcal{T}_h}\langle(\textbf{q}-\mathbb{Q}_h\textbf{q})\cdot \mathbf{n},{\S}(v_{b})-v_{b}\rangle_{\partial T} \\
&+\rho h^{-1}\sum_{T\in\mathcal{T}_h}\langle Q_{b}\S(Q_{b}u)-Q_{b}u,Q_{b}{\S}(v_{b})-v_{b}\rangle_{\partial T}
\end{split}
\end{equation}
is a linear functional on $V_{b}$.
\end{lemma}
\begin{proof}
The proof is similar to the proof of Lemma 5.1 in \cite{sup_LWW2018}, and therefore the details are
omitted here.
\end{proof}

\section{Technical Estimates}\label{Section:TE:22:45}
We consider the second order elliptic model problem \eqref{model} on the unit cubic domain $\Omega=(0, 1)^{3}$. Let the domain $\Omega$ be partitioned into cubic elements as the Cartesian product of three partitions $\Delta_{x}$, $\Delta_{y}$ and $\Delta_{z}$ on the unit interval $(0, 1)$:
\begin{eqnarray*}
&&\Delta_{x}:0=x_{0}<x_{1}<x_{2}\ldots <x_{i}<\ldots< x_{n-1}<x_{n}=1,\\
&&\Delta_{y}:0=y_{0}<y_{1}<y_{2}\ldots <y_{j}<\ldots< y_{m-1}<y_{m}=1,\\
&&\Delta_{z}:0=z_{0}<z_{1}<z_{2}\ldots< z_{s} < \ldots<z_{q-1}<z_{q}=1.
\end{eqnarray*}
Let $T=[x_{i-1}, x_{i}]\times[y_{j-1}, y_{j}]\times[z_{s-1}, z_{s}] \in {\cal T}_h$ be a cubic element for $i=1, \ldots, n$, $j=1,\ldots, m$ and $s=1, \ldots, q$ (see Figure
6.1 for reference).
Denote by $|e_x|$, $|e_y|$ and $|e_z|$ the length of the edge of the cubic element $T$ in the $x$-, $y$- and $z$- direction, respectively. Denote by $|T|$ the volume of the element $T$. Denote by $|F_p|$ the area of the flat face $F_{p}$ for $p=1, \ldots, 6$ such that $|F_{1}|=|F_{2}|$, $|F_{3}|=|F_{4}|$ and $|F_{5}|=|F_{6}|$. Denote by $M_{c}=(x_{c}, y_{c}, z_{c})$ the center of the cubic element $T$, and $M_{p}=(x_{p}^{*}, y_{p}^{*}, z_{p}^{*})$ the center of the flat face $F_{p}$ for $p=1, \ldots, 6$, respectively.
The unit outward normal directions to the flat faces $F_p$ for $p=1, \ldots, 6$ are given by $\bn_1=(-1, 0, 0)^{'}$, $\bn_2=(1,0,0)^{'}$, $\bn_3=(0,-1,0)^{'}$, $\bn_4=(0,1,0)^{'}$, $\bn_5=(0,0,-1)^{'}$, and $\bn_6=(0,0,1)^{'}$, respectively.

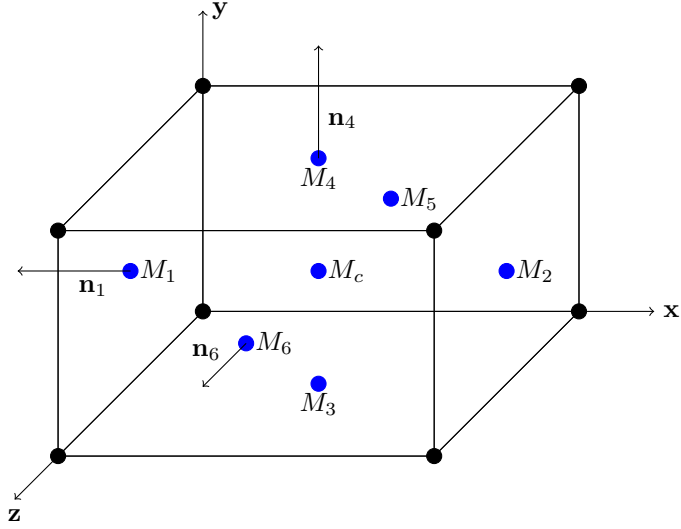
\begin{figure}[h]
\begin{center}
\begin{tikzpicture}
\coordinate (A1) at (-2,0,-2); \filldraw[black] (A1) circle(0.1);
\coordinate (A2) at (3, 0,-2); \filldraw[black] (A2) circle(0.1);
\coordinate (A3) at (3, 3,-2);\filldraw[black] (A3) circle(0.1);
\coordinate (A4) at (-2,3,-2); \filldraw[black] (A4) circle(0.1);
\coordinate (A5) at (-2,0,3); \filldraw[black] (A5) circle(0.1);
\coordinate (A6) at (3, 0,3); \filldraw[black] (A6) circle(0.1);
\coordinate (A7) at (3, 3,3); \filldraw[black] (A7) circle(0.1);
\coordinate (A8) at (-2,3,3); \filldraw[black] (A8) circle(0.1);

 \coordinate (M1234mid) at (0.5,1.5,-2.0);\coordinate (M5678mid) at (0.5,1.5,3.0);
\coordinate (M2376mid) at (3.0,1.5,0.5);\coordinate (M1485mid) at (-2.0,1.5,0.5);
\coordinate (M1562mid) at (0.5,0.0,0.5);\coordinate (M3487mid) at (0.5,3.0,0.5);
 \filldraw[blue] (M1485mid)circle(0.1);
 \filldraw[blue] (M2376mid)circle(0.1);

\filldraw[blue] (M3487mid)circle(0.1);
 \filldraw[blue] (M1562mid)circle(0.1);

 \filldraw[blue] (M1234mid)circle(0.1);
 \filldraw[blue] (M5678mid)circle(0.1);

\draw (A1)--(A4)--(A8)--(A5)--cycle;
\draw (A2)--(A3)--(A7)--(A6)--cycle;
\draw (A1)--(A5)--(A6)--(A2)--cycle;
\draw (A3)--(A4)--(A8)--(A7)--cycle;
\draw (A1)--(A2)--(A3)--(A4)--cycle;
\draw (A5)--(A6)--(A7)--(A8)--cycle;

\coordinate (M12end) at (4,0,-2);
\coordinate (M14end) at (-2,4,-2);
\coordinate (M15end) at (-2,0.0,4.5);
\draw[->] (A2)--(M12end);
\draw[->] (A4)--(M14end);
\draw[->] (A5)--(M15end);
\draw node[right] at (M12end) {$\bx$};
\draw node[right] at (M14end) {$\by$};
\draw node[below] at (M15end) {$\bz$};

\coordinate (center) at (0.5,1.5,0.5);
\filldraw[blue] (center)circle(0.1);

 \draw node [right][black] at (M1485mid) {$M_1$};
\draw node [right][black] at (M5678mid) {$M_6$};
\draw node [below][black] at (M3487mid) {$M_4$};
\draw node [right][black] at (M1234mid) {$M_5$};
\draw node [below][black] at (M1562mid) {$M_3$};
\draw node [right][black] at (M2376mid) {$M_2$};

\draw node[right] at (center) {$M_c$};

\coordinate (M1end) at (-3.5,1.5,0.5);
\coordinate (M4end) at (0.5,4.5,0.5);
\coordinate (M6end) at (0.5,1.5,4.5);
\coordinate (M5end) at (0.5,1.5,-3.5);
\draw[->] (M1485mid)--(M1end);
\coordinate (ne1) at (-2.5,1.5,0.5);
\draw node[below] at (ne1) {$\bn_1$};
\draw[->] (M3487mid)--(M4end);
\coordinate (ne4) at (0.5,3.5,0.5);
\draw node[right] at (ne4) {$\bn_4$};

\coordinate (ne6) at (0.4,1.5,3.3);
\draw node[left] at (ne6) {$\bn_6$};
\draw[->] (M5678mid)--(M6end);
\end{tikzpicture}
\caption{A cubic element $T\in\T_h$.}\label{cubic-element}
\end{center}
\end{figure}

On the element $T$, denote by $v_{bp}$ the value of $v_b$ on the face $F_p, \ p=1,\ldots, 6$. Using \eqref{disgradient}, we have
$$
(\nabla_d v_b, \boldsymbol \psi)_T = \langle v_b, \boldsymbol\psi\cdot\bn\rangle_\pT,\qquad \forall \boldsymbol\psi \in [P_0(T)]^3,
$$
which gives
\begin{equation}\label{EQ:weak-gradient}
\nabla_d v_b = \left(\frac{v_{b2}-v_{b1}}{|e_x|},\frac{v_{b4}-v_{b3}}{|e_y|},\frac{v_{b6}-v_{b5}}{|e_z|}\right)^\prime.
\end{equation}

For any linear function $\psi\in P_1(T)$, it is easy to see that $Q_{b}\psi=\psi(M_{p})$ on each face $F_{p}$. It thus follows from \eqref{a7} that
\begin {eqnarray}\label{a11}
\sum_{p=1}^{6}|F_{p}| \,{\S}(v_{b})(M_{p})\psi(M_{p})=\sum_{p=1}^{6}|F_{p}|\: v_{bp}\:\psi(M_{p}),\qquad \forall\psi\in P_{1}(T).
\end {eqnarray}

\begin{lemma}\label{lemma2.2}
Assume a cubic element $T=[x_{i-1},x_{i}]\times[y_{j-1},y_{j}]\times[z_{s-1},z_{s}]\in {\cal T}_h$. Let the extension function $\S(v_{b})\in P_{1}(T)$ be defined in \eqref{a11}. There holds
\begin{equation}\label{a12}
\begin{split}
({\S}(v_{b})-v_{b})(M_{1}) =&({\S}(v_{b})-v_{b})(M_{2}) \\
=&\frac{|F_{3}|(v_{b3}+v_{b4})+|F_{5}|(v_{b5}+v_{b6})-(|F_{3}|+|F_{5}|)(v_{b1}+v_{b2})}{2(|F_{1}|+|F_{3}|+|F{5}|)},
\end{split}
\end{equation}
\begin{equation}\label{a13}
\begin{split}
({\S}(v_{b})-v_{b})(M_{3})=&({\S}(v_{b})-v_{b})(M_{4})  \\
=&\frac{|F_{1}|(v_{b1}+v_{b2})+|F_{5}|(v_{b5}+v_{b6})-(|F_{1}|+|F_{5}|)(v_{b3}+v_{b4})}{2(|F_{1}|+|F_{3}|+|F_{5}|)},
\end{split}
\end{equation}
\begin{equation}\label{a14}
\begin{split}
({\S}(v_{b})-v_{b})(M_{5})=&({\S}(v_{b})-v_{b})(M_{6})\\
=&\frac{|F_{1}|(v_{b1}+v_{b2})+|F_{3}|(v_{b3}+v_{b4})-(|F_{1}|+|F_{3}|)(v_{b5}+v_{b6})}{2(|F_{1}|+|F_{3}|+|F_{5}|)}.
\end{split}
\end{equation}
Furthermore, there holds
\begin{equation}\label{a15}
\begin{split}
|F_{1}|({\S}(v_{b})-v_{b})(M_{1})+|F_{3}|({\S}(v_{b})-v_{b})(M_{3})+|F_{5}|({\S}(v_{b})-v_{b})(M_{5})=0.
\end{split}
\end{equation}
\end{lemma}
\begin{proof}
From the definition of the extension function ${\S}(v_{b})$, we have
\begin{equation}\label{extension-definition}
{\S}(v_{b})=c_{1}+c_{2}(x-x_{c})+c_{3}(y-y_{c})+c_{4}(z-z_{c}).
\end{equation}
Letting $\psi=1$ in \eqref{a11} gives
\begin{equation*}
\begin{split}
\sum_{p=1}^{6}\mid F_{p}\mid \Big(c_{1}+c_{2}(x_{p}^{*}-x_{c})+c_{3}(y_{p}^{*}-y_{c})+c_{4}(z_{p}^{*}-z_{c})\Big)=\sum_{p=1}^{6}\mid F_{p}\mid v_{bp},
\end{split}
\end{equation*}
which leads to
$$c_{1}=\frac{|F_{1}|(v_{b1}+v_{b2})+|F_{3}|(v_{b3}+v_{b4})+|F_{5}|(v_{b5}+v_{b6})}{2(|F_{1}|+|F_{3}|+|F_{5}|)}.$$
Similarly, setting $\psi=x-x_{c}$, $\psi=y-y_{c}$, and $\psi=z-z_{c}$ in \eqref{a11}  yields
$$c_{2}=\frac{v_{b2}-v_{b1}}{|e_x|}, \qquad  c_{3}=\frac{v_{b4}-v_{b3}}{|e_y|},\qquad c_{4}=\frac{v_{b6}-v_{b5}}{|e_z|}.$$
Therefore,  \eqref{extension-definition} can be rewriten as follows
\begin {eqnarray*}
{\S}(v_{b})&=&\frac{|F_{1}|(v_{b1}+v_{b2})+|F_{3}|(v_{b3}+v_{b4})+|F_{5}|(v_{b5}+v_{b6})}{2(|F_{1}|+|F_{3}|+|F_{5}|)}
+\frac{v_{b2}-v_{b1}}{|e_x|}(x-x_{c})\\&&+\frac{v_{b4}-v_{b3}}{|e_y|}(y-y_{c})
+\frac{v_{b6}-v_{b5}}{|e_z|}(z-z_{c}).
\end {eqnarray*}

Next, we compute ${\S}(v_{b})-v_{b}$ at the center $M_{p}$ of each flat face $F_{p}$ for $p=1, \cdots, 6$. At the center $M_{1}$ of the flat face $F_{1}$, we have
 \begin{eqnarray*}
\begin{split}
 ({\S}(v_{b})-v_{b})(M_{1})=&\frac{|F_{1}|(v_{b1}+v_{b2})+|F_{3}|(v_{b3}+v_{b4})+|F_{5}|(v_{b5}+v_{b6})}{2(|F_{1}|+|F_{3}|+|F_{5}|)}
-\frac{(v_{b2}-v_{b1})}{|e_x|}\frac{|e_x|}{2}-v_{b1} \\
=&\frac{|F_{3}|(v_{b3}+v_{b4})+|F_{5}|(v_{b5}+v_{b6})-(|F_{3}|+|F_{5}|)(v_{b1}+v_{b2})}{2(|F_{1}|+|F_{3}|+|F_{5}|)}.
\end{split}
 \end{eqnarray*}
Similarly, we obtain
\begin{equation*}
\begin{split}
 ({\S}(v_{b})-v_{b})(M_{2})
=&\frac{|F_{3}|(v_{b3}+v_{b4})+|F_{5}|(v_{b5}+v_{b6})-(|F_{3}|+|F_{5}|)(v_{b1}+v_{b2})}{2(|F_{1}|+|F_{3}|+|F_{5}|)},
\end{split}
\end{equation*}
 \begin{equation*}
\begin{split}
 ({\S}(v_{b})-v_{b})(M_{3})
=&\frac{|F_{1}|(v_{b1}+v_{b2})+|F_{5}|(v_{b5}+v_{b6})-(|F_{1}|+|F_{5}|)(v_{b3}+v_{b4})}{2(|F_{1}|+|F_{3}|+|F_{5}|)},
\end{split}
\end{equation*}
\begin{equation*}
\begin{split}
 ({\S}(v_{b})-v_{b})(M_{4})
=&\frac{|F_{1}|(v_{b1}+v_{b2})+|F_{5}|(v_{b5}+v_{b6})-(|F_{1}|+|F_{5}|)(v_{b3}+v_{b4})}{2(|F_{1}|+|F_{3}|+|F_{5}|)},
\end{split}
\end{equation*}
 \begin{equation*}
\begin{split}
 ({\S}(v_{b})-v_{b})(M_{5})
=&\frac{|F_{1}|(v_{b1}+v_{b2})+|F_{3}|(v_{b3}+v_{b4})-(|F_{3}|+|F_{1}|)(v_{b5}+v_{b6})}{2(|F_{1}|+|F_{3}|+|F_{5}|)},
\end{split}
\end{equation*}
\begin{equation*}
\begin{split}
 ({\S}(v_{b})-v_{b})(M_{6})
=&\frac{|F_{1}|(v_{b1}+v_{b2})+|F_{3}|(v_{b3}+v_{b4})-(|F_{3}|+|F_{1}|)(v_{b5}+v_{b6})}{2(|F_{1}|+|F_{3}|+|F_{5}|)}.
\end{split}
\end{equation*}
This completes the proof of the Lemma.
\end{proof}

We now focus on the two terms on the right-hand side of the error equation  \eqref{a9}, where the first term $\sum_{T\in\mathcal{T}_h}\langle(\textbf{q}-\mathbb{Q}_h\textbf{q})\cdot \mathbf{n},{\S}(v_{b})-v_{b}\rangle_{\partial T}$ is critical in the analysis. Lemma \ref{lemma2.2} indicates that ${\S}(v_{b})-v_{b}$ has the same value at the center of the flat faces $F_{1}$ and $F_{2}$. Moreover, ${\S}(v_{b})-v_{b}$ has the same directional derivative along the flat faces $F_{1}$ and $F_{2}$ which are $\frac{\partial({\S}(v_{b})-v_b)}{\partial y}=\nabla_d v_b \cdot \bn_4$ and $\frac{\partial({\S}(v_{b})-v_b)}{\partial z}=\nabla_d v_b \cdot \bn_6$, respectively. Hence, ${\S}(v_{b})-v_{b}$ has the same value along the flat faces $F_{1}$ and $F_{2}$ at the symmetric points $(x_{i-1}, y, z)$ and $(x_{i}, y, z)$. Likewise, ${\S}(v_{b})-v_{b}$ has the same value along the flat faces $F_{3}$ and $F_{4}$ at the symmetric points $(x, y_{j-1}, z)$ and $(x, y_{j}, z)$, and has the same value along the flat faces $F_{5}$ and $F_{6}$ at the symmetric points $(x, y, z_{s-1})$ and $(x, y, z_{s})$, respectively. It thus follows that
  \begin{eqnarray}\label{a16}
 \langle\mathbb{Q}_h\textbf{q}\cdot \mathbf{n},{\S}(v_{b})-v_{b}\rangle_{\partial T}=0.
\end{eqnarray}

Since ${\S}(v_{b})-v_{b}$ has the same value at the symmetric points $(x_{i-1}, y, z)$ and $(x_{i}, y, z)$, this boundary function on the flat faces $F_{1}$ and $F_{2}$ can be extended to the cubic element $T$ by assigning the value $({\S}(v_{b})-v_{b})(x_{i-1}, y, z)$ along each face parallel to the flat face $F_1$ (or $F_2$). Denote this extension by $\chi_{1}$; i.e.,
\begin{equation}\label{a17}
\chi_{1}(x, y, z):=({\S}(v_{b})-v_{b})(x_{i-1}, y, z),\qquad (x, y, z)\in T.
\end{equation}
Similarly, denote by $\chi_{2}$ the extension function of ${\S}(v_{b})-v_{b}$ to the cubic element $T$ by assuming the value $({\S}(v_{b})-v_{b})(x, y_{j-1}, z)$ along each face parallel to the flat face $F_3$ (or $F_4$); i.e.,
\begin{equation}\label{a18}
\begin{split}
\chi_{2}(x, y, z):=({\S}(v_{b})-v_{b})(x, y_{j-1}, z),\qquad (x, y, z)\in T.
\end{split}
\end{equation}
Likewise, we define
\begin{equation}\label{a19}
\begin{split}
\chi_{3}(x, y, z):=({\S}(v_{b})-v_{b}) (x, y, z_{s-1}), \qquad (x, y, z)\in T.
\end{split}
\end{equation}
It follows from Lemma \ref{lemma2.2} and \eqref{EQ:weak-gradient} that
 \begin{equation}\label{FPS}
\begin{split}
&\partial_{x}\chi_{1}=0, \quad \partial_{y}\chi_{1}=\nabla_d v_b \cdot \bn_4,\quad\partial_{z}\chi_{1}=\nabla_d v_b \cdot \bn_6,\\
& \partial_{y}\chi_{2}=0,\quad\partial_{x}\chi_{2}=\nabla_d v_b \cdot \bn_2,\quad \partial_{z}\chi_{2}=\nabla_d v_b \cdot \bn_6,\\
&\partial_{z}\chi_{3}=0,\quad\partial_{x}\chi_{3}=\nabla_d v_b \cdot \bn_2, \quad \partial_{y}\chi_{3}=\nabla_d v_b \cdot \bn_4,\\
&|F_{1}|\chi_{1}(M_{1})+|F_{3}|\chi_{2}(M_{3})+|F_{5}|\chi_{3}(M_{5})=0.\end{split}
\end{equation}

\begin{lemma}\label{lemm1.1}
Assume $u\in H^{3}(\Omega)$ is a given function. Let ${\cal T}_{h}=\Delta_{x}\times\Delta_{y}\times\Delta_{z}$ be a cubic partition. On each element $T\in {\cal T}_h$, for any $v_{b}\in V_{b}^0$, there holds
\begin{equation}\label{a21}
\begin{split}
&\langle{\color{black}{(\textbf{q}-\mathbb{Q}_h\textbf{q})}} \cdot \mathbf{n}, {\S}(v_{b})	-v_{b}\rangle_{\partial T} \\
=&{\color{black}{\chi_{1}(M_{1})\int_{T}q_{1x}dT+\chi_{2}(M_{3})\int_{T}q_{2y}dT+\chi_{3}(M_{5})\int_{T}q_{3z}dT}}+R_{1}(T),
\end{split}
\end{equation}
where $\chi_{i}$ $(i=1, 2, 3)$ are the extension functions defined in (\ref{a17})-(\ref{a19}), and
$\textbf{q}=A\nabla u=(q_1, q_2, q_3)'$, $q_1=a_{11}u_x+a_{12}u_y+a_{13}u_z$, $q_2=a_{21}u_x+a_{22}u_y+a_{23}u_z$, $q_3=a_{31}u_x+a_{32}u_y+a_{33}u_z$, $q_{1x}=\frac{\partial q_1}{\partial x}$, $q_{2y}=\frac{\partial q_2}{\partial y}$, $q_{3z}=\frac{\partial q_3}{\partial z}$, respectively.
The remainder term $R_{1}(T)$ satisfies the following estimate
\begin{eqnarray}\label{RE-1-1}
\sum_{T\in{\cal T}_{h}}|R_{1}(T)|\leq C h^{2} \|\textbf{q}\|_2 \|\nabla_d v_b  \|_0.
\end{eqnarray}

\end{lemma}
\begin{proof}
Using the definition of $\chi_{i}$ in (\ref{a17})-(\ref{a19}), \eqref{a16}, and the usual integration by parts, we obtain
\begin{equation}\label{a20}
\begin{split}
&\langle{\color{black}{ (\textbf{q}-\mathbb{Q}_h\textbf{q})}}\cdot \mathbf{n},{\S}(v_{b})-v_{b}\rangle_{\partial T} \\
=&\langle{\color{black}{\textbf{q}}}\cdot \bn,{\S}(v_{b})-v_{b}\rangle_{\partial T} \\
=&-\int_{F_{1}}{\color{black}{q_1}}\chi_{1}dF+\int_{F_{2}}{\color{black}{q_1}}\chi_{1}dF-\int_{F_{3}}{\color{black}{q_2}}\chi_{2}dF\\&+\int_{F_{4}}{\color{black}{q_2}}\chi_{2}dF
-\int_{F_{5}}{\color{black}{q_3}}\chi_{3}dF+\int_{F_{6}}{\color{black}{q_3}}\chi_{3}dF \\
=&{\color{black}{\int_{T}q_{1x}\chi_{1}dT+\int_{T}q_{2y}\chi_{2}dT+\int_{T}q_{3z}\chi_{3}dT}}.
\end{split}
\end{equation}

Since $\chi{_{1}}$ is linear in both the $y$- direction and the $z$- direction, and is constant in the $x$- direction, we have
$$\chi{_{1}}(y,z)=\chi{_{1}}(M_{1})+\partial_{y}\chi{_{1}}(y-y_{c})+\partial_{z}\chi{_{1}}(z-z_{c}),$$
which, together with the usual integration by parts, gives
\begin{eqnarray*}
 \int_{T}q_{1x}\chi_{1}dT&=&\int_{T}q_{1x}\chi_{1}(M_{1})dT+\int_{T}q_{1x}\partial_{y}\chi_{1}(y-y_{c})dT+
 \int_{T}q_{1x}\partial_{z}\chi_{1}(z-z_{c})dT\\
 &=&\int_{T}q_{1x}\chi_{1}(M_{1})dT+\int_{T}q_{1xy}\partial_{y}\chi_{1}E_{31}(y)dT+
 \int_{T}q_{1xz}\partial_{z}\chi_{1}E_{32}(z)dT,
 \end{eqnarray*}
where $E_{31}(y)=\frac{1}{8}|e_y|^{2}-\frac{1}{2}(y-y_{c})^2$ and $E_{32}(z)=\frac{1}{8}|e_z|^{2}-\frac{1}{2}(z-z_{c})^2$.

Similarly, there holds
\begin{eqnarray*}
 \int_{T}q_{2y}\chi_{2}dT&=&\int_{T}q_{2y}\chi_{2}(M_{3})dT+\int_{T}q_{2y}\partial_{x}\chi_{2}(x-x_{c})dT+
 \int_{T}q_{2y}\partial_{z}\chi_{2}(z-z_{c})dT\\
 &=&\int_{T}q_{2y}\chi_{2}(M_{3})dT+\int_{T}q_{2yx}\partial_{x}\chi_{2}E_{41}(x)dT+
 \int_{T}q_{2yz}\partial_{z}\chi_{2}E_{42}(z)dT,
 \end{eqnarray*}
where $E_{41}(x)=\frac{1}{8}|e_x|^{2}-\frac{1}{2}(x-x_{c})^2$, $E_{42}(z)=\frac{1}{8}|e_z|^{2}-\frac{1}{2}(z-z_{c})^2$.

Likewise, we have
\begin{eqnarray*}
 \int_{T}q_{3z}\chi_{3}dT&=&\int_{T}q_{3z}\chi_{3}(M_{5})dT+\int_{T}q_{3z}\partial_{x}\chi_{3}(x-x_{c})dT+
 \int_{T}q_{3z}\partial_{y}\chi_{3}(y-y_{c})dT\\
 &=&\int_{T}q_{3z}\chi_{3}(M_{5})dT+\int_{T}q_{3zx}\partial_{x}\chi_{3}E_{51}(x)dT+
 \int_{T}q_{3zy}\partial_{y}\chi_{3}E_{52}(y)dT,
 \end{eqnarray*}
 where $E_{51}(x)=\frac{1}{8}|e_x|^{2}-\frac{1}{2}(x-x_{c})^2$,  and $E_{52}(y)=\frac{1}{8}|e_y|^{2}-\frac{1}{2}(y-y_{c})^2$.

Substituting the above three identities into \eqref{a20} gives
{\color{black}{\begin{eqnarray*}
&&\langle(\textbf{q}-\mathbb{Q}_h\textbf{q})\cdot \mathbf{n},{\S}(v_{b})-v_{b}\rangle_{\partial T}\nonumber \\
&=&\chi_{1}(M_{1})\int_{T}q_{1x}dT+\chi_{2}(M_{3})\int_{T}q_{2y}dT+\chi_{3}(M_{5})\int_{T}q_{3z}dT+R_{1}(T),
\end{eqnarray*}}}
where the remainder term $R_{1}(T)$ is given by
\begin{equation}\label{r1}
\begin{split}
R_{1}(T)=&\int_{T}q_{1xy}\partial_{y}\chi_{1}E_{31}(y)dT+\int_{T}q_{1xz}\partial_{z}\chi_{1}E_{32}(z)dT\\
&+\int_{T}q_{2yx}\partial_{x}\chi_{2}E_{41}(x)dT+
 \int_{T}q_{2yz}\partial_{z}\chi_{2}E_{42}(z)dT\\
&+\int_{T}q_{3zx}\partial_{x}\chi_{3}E_{51}(x)dT+
 \int_{T}q_{3zy}\partial_{y}\chi_{3}E_{52}(y)dT.
\end{split}
\end{equation}
Using the Cauchy-Schwarz inequality and \eqref{FPS}, there holds
\begin{eqnarray*}
\Big|\sum_{T\in {\cal T}_h}\int_{T}q_{1xy}\partial_{y}\chi{_{1}}E_{31}(y)dT\Big|&\leq& C h^2 (\sum_{T\in {\cal T}_h}\|\nabla^2q_1\|^2_{T})^{\frac{1}{2}}(\sum_{T\in {\cal T}_h}\|\partial_{y}\chi{_{1}}\|^2_{T})^{\frac{1}{2}} \\
&\leq&C h^2 \| q_1\|_2 \|\nabla_d v_b  \|_0.
\end{eqnarray*}
Each of the rest five terms in the remainder term $R_{1}(T)$ in (\ref{r1}) can be estimated in a similar way. This completes the proof of the Lemma.
\end{proof}

The following Lemma shall provide an estimate for the second term on the right-hand side of the error equation \eqref{a9}.

\begin{lemma}\label{lemma3.3}
Under the assumptions of Lemma \ref{lemm1.1}, there holds
\begin{equation}\label{a26}
\begin{split}
&\rho h^{-1} \langle Q_{b}{\S}(Q_{b}{\color{black}{u}})-Q_{b}{\color{black}{u}},Q_{b}{\S}(v_{b})-v_{b}\rangle_{\partial T} \\
=&-{\color{black}{A_{1}\rho  h^{-1}}} (|e_x|\chi_{1}(M_{1})\int_{T}{\color{black}{u_{xx}}}dT+|e_y|\chi_{2}(M_{3})\int_{T}{\color{black}{u_{yy}}}dT\\&+
|e_z|\chi_{3}(M_{5})\int_{T}{\color{black}{u_{zz}}}dT )+R_{2}(T),
\end{split}
\end{equation}
where $A_1=\frac{1}{6}$ and the remainder term $R_{2}(T)$ satisfies
\begin{eqnarray}\label{RE-1-2}
\sum_{T\in{\cal T}_{h}}|R_{2}(T)|\leq C h^{2} \| u\|_{3} \3bar{\S}(v_{b})-v_{b}\3bar_{\mathcal{E}_h}.
\end{eqnarray}
\end{lemma}
Here we define
\begin{eqnarray}\label{RE-2018-6-3}
\3bar{\S}(v_{b})-v_{b}\3bar^{2}_{\mathcal{E}_h}:=\rho h^{-1}\sum_{T\in{\cal T}_{h}}\langle Q_{b}{\S}(v_{b})-v_{b},Q_{b}{\S}(v_{b})-v_{b}\rangle_{\partial T}.
\end{eqnarray}
\begin{proof}
Using \eqref{a7}, (\ref{a17})-(\ref{a19}), Lemma \ref{lemma2.2}, and (\ref{FPS}), we have
\begin{equation}\label{a22}
\begin{split}
&\rho h^{-1}\langle Q_{b}{\S}(Q_{b}u)-Q_{b}u,Q_{b}{\S}(v_{b})-v_{b}\rangle_{\partial T} \\
=&-\rho h^{-1}\langle Q_{b}u,{\S}(v_{b})-v_{b}\rangle_{\partial T} \\
=&-\rho h^{-1}\Big(|F_{1}|\chi_{1}(M_{1})Q_{b}u(M_1)+|F_{2}|\chi_{1}(M_{2})Q_{b}u(M_2)\\&+|F_{3}|\chi_{2}(M_{3})Q_{b}u(M_3)
+|F_{4}|\chi_{2}(M_{4})Q_{b}u(M_4)\\&+|F_{5}|\chi_{3}(M_{5})Q_{b}u(M_5)+|F_{6}|\chi_{3}(M_{6})Q_{b}u(M_6)\Big) \\
=&-\rho h^{-1}|F_{1}|\chi_{1}(M_{1})\Big(Q_{b}u(M_1)+Q_{b}u(M_2)-Q_{b}u(M_5)-Q_{b}u(M_6)\Big) \\
&-\rho h^{-1}|F_{3}|\chi_{2}(M_{3})\Big(Q_{b}u(M_3)+Q_{b}u(M_4)-Q_{b}u(M_5)-Q_{b}u(M_6)\Big).\\
\end{split}
\end{equation}
Note that $Q_{b}u|_{F_i}=\frac{1}{|F_i|}\int_{F_i}udF$ is the average of $u$ on the flat face $F_i$. Using the Euler-MacLaurin formula gives
\begin{equation}\label{a23}
\begin{split}
&|F_{1}|\Big(Q_{b}u(M_{1})+Q_{b}u(M_{2})-Q_{b}u(M_{5})-Q_{b}u(M_{6})\Big) \\
=&\int_{F_1} u(x_{i-1},y,z)dF+\int_{F_2}u(x_{i},y,z)dF \\
&-\frac{|F_1|}{|F_5|}\Big(\int_{F_5}u(x,y, z_{s-1})dF +\int_{F_6}u(x,y, z_{s})dF\Big)\\
=&\frac{1}{|e_x|}\Big(2\int_{T}u(x,y,z)dT+A_{1}|e_x|^{2}\int_{T}u_{xx}dT+\frac{1}{24}|e_x|^{3}\int_{T}u_{xxx}E_{1}(x)dT  \Big)\\
&-\frac{1}{|e_z|}\cdot\frac{|F_1|}{|F_5|}\Big(2\int_{T}u(x,y,z)dT+A_{1}|e_z|^{2}\int_{T}u_{zz}dT\\
&+\frac{1}{24}|e_z|^{3}\int_{T}u_{zzz}E_{3}(z)dT\Big)\\
=&A_1\Big(|e_x|\int_{T}u_{xx}dT-\frac{|e_z|^{2}}{|e_x|}\int_{T}u_{zz}dT\Big)\\
&+\frac{1}{24}\Big(|e_x|^{2}\int_{T}u_{xxx}E_{1}(x)dT-\frac{|e_z|^{3}}{|e_x|}u_{zzz}E_{3}(z)dT\Big),
\end{split}
\end{equation}
where $E_{1}(x)$ and $E_{3}(z)$ are the cubic polynomials in both the $x$- direction and the $z$- direction.

Similarly, we arrive at
\begin{equation}\label{a24}
\begin{split}
&|F_{3}|\Big(Q_{b}u(M_{3})+Q_{b}u(M_{4})-Q_{b}u(M_{5})-Q_{b}u(M_{6})\Big) \\
=&A_{1}|e_y|\int_{T}u_{yy}dT-A_{1}\frac{|e_z|^{2}}{|e_y|}\int_{T}u_{zz}dT
\\&+\frac{{|e_y|} ^{2}}{24}\int_{T}u_{yyy}E_{2}(y)dT-\frac{|e_z|^{3}}{24|e_y|}\int_{T}u_{zzz}E_{3}(z)dT,
\end{split}
\end{equation}
where $E_{2}(y)$ is the cubic polynomial in the $y$- direction.

Substituting \eqref{a23} - \eqref{a24} into \eqref{a22} and using (\ref{FPS}), we have
\begin{equation*}
\begin{split}
&\rho h^{-1}\langle Q_{b}{\S}(Q_{b}u)-Q_{b}u,Q_{b}{\S}(v_{b})-v_{b}\rangle_{\partial T} \\
=&-\rho h^{-1}\chi_{1}(M_{1})\Big(A_{1}|e_x|\int_{T}u_{xx}dT-A_{1}\frac{|e_z|^{2}}{|e_x|}\int_{T}u_{zz}dT\\&+
\frac{|e_x|^{2}}{24}\int_{T}u_{xxx}E_{1}(x)dT  -\frac{|e_z|^{3}}{24|e_x|}\int_{T}u_{zzz}E_{3}(z)dT\Big)\\&-\rho h^{-1}\chi_{2}(M_{3})\Big(A_{1}|e_y|\int_{T}u_{yy}dT-A_{1}\frac{|e_z|^{2}}{|e_y|}\int_{T}u_{zz}dT \\& +
\frac{|e_y|^{2}}{24}\int_{T}u_{yyy}E_{2}(y)dT-\frac{|e_z|^{3}}{24|e_y|}\int_{T}u_{zzz}E_{3}(z)dT\Big) \\
=&-A_{1}\rho h^{-1}\Big(|e_x|\chi_{1}(M_{1})\int_{T}u_{xx}dT+|e_y|\chi_{2}(M_{3})\int_{T}u_{yy}dT\\&+|e_z|\chi_{3}(M_{5})\int_{T}u_{zz}dT\Big) +R_{2}(T),
\end{split}
\end{equation*}
where the remainder $R_{2}(T)$ is given by
\begin{eqnarray*}
R_{2}(T)&=&-\frac{1}{24}\rho h^{-1}\Big(|e_{x} |^{2}\chi_{1}(M_{1})\int_{T}u_{xxx}E_{1}(x)dT
\\ &&+|e_{y}|^{2}\chi_{2}(M_{3})\int_{T}u_{yyy}E_{2}(y)dT +|e_{z}|^{2}\chi_{3}(M_{5})\int_{T}u_{zzz}E_{3}(z)dT\Big).
\end{eqnarray*}

Similar to the proof of \eqref{RE-1-1}, it is easy to arrive at \eqref{RE-1-2}. This completes the proof of the Lemma.
\end{proof}

\section{Superconvergence}\label{Section:SWG}
In this section, we shall establish the superconvergence error estimates for the simplified weak Galerkin finite element scheme (\ref{SWG}) for solving the three dimensional second order model problem \eqref{model} on the cubic partitions.

\begin{theorem}\label{superconvergence-H-1-1}
Assume that $u\in H^{3}(\Omega)$ is the exact solution of the second order elliptic model problem \eqref{model} in three dimensions. Let $u_{b}\in V_{b}$ be the weak Galerkin finite element solution arising from the simplified WG scheme \eqref{SWG} satisfying the boundary condition $u_{b}=Q_{b}g$ on $\partial\Omega$. On each cubic element $T    =[x_{i-1},x_{i}]\times[y_{j-1},y_{j}]\times [z_{s-1},z_{s}] \in {\cal T}_h$, we define $w_{b}\in V_{b}$ as follows
\begin{equation*}
w_{b}=\left\{
\begin{array}{lr}
      \frac{1}{12}\rho^{-1}h^{-1}\left(\rho h^{-1}(|e_y|^{2}Q_{b}u_{yy}|_{F_{1}}+|e_z|^{2}Q_{b}u_{zz}|_{F_{1}})-
      6(|e_y|Q_{b}q_{2y}|_{F_{1}}+|e_z|Q_{b}q_{3z}|_{F_{1}})\right),  \mbox{on } F_{1},\\
     \frac{1}{12}\rho^{-1}h^{-1}\left(\rho h^{-1}(|e_y|^{2}Q_{b}u_{yy}|_{F_{2}}+|e_z|^{2}Q_{b}u_{zz}|_{F_{2}})-
      6(|e_y|Q_{b}q_{2y}|_{F_{2}}+|e_z|Q_{b}q_{3z}|_{F_{2}})\right), \mbox{on } F_{2},\\
      \frac{1}{12}\rho^{-1}h^{-1}\left(\rho h^{-1}(|e_x|^{2}Q_{b}u_{xx}|_{F_{3}}+|e_z|^{2}Q_{b}u_{zz}|_{F_{3}})-
      6(|e_x|Q_{b}q_{1x}|_{F_{3}}+|e_z|Q_{b}q_{3z}|_{F_{3}})\right), \mbox{on } F_{3},\\
     \frac{1}{12}\rho^{-1}h^{-1}\left(\rho h^{-1}(|e_x|^{2}Q_{b}u_{xx}|_{F_{4}}+|e_z|^{2}Q_{b}u_{zz}|_{F_{4}})-
      6(|e_x|Q_{b}q_{1x}|_{F_{4}}+|e_z|Q_{b}q_{3z}|_{F_{4}})\right),  \mbox{on } F_{4},\\
      \frac{1}{12}\rho^{-1}h^{-1}\left(\rho h^{-1}(|e_y|^{2}Q_{b}u_{yy}|_{F_{5}}+|e_x|^{2}Q_{b}u_{xx}|_{F_{5}})-
      6(|e_y|Q_{b}q_{2y}|_{F_{5}}+|e_x|Q_{b}q_{1x}|_{F_{5}})\right),   \mbox{on } F_{5},\\
       \frac{1}{12}\rho^{-1}h^{-1}\left(\rho h^{-1}(|e_y|^{2}Q_{b}u_{yy}|_{F_{6}}+|e_x|^{2}Q_{b}u_{xx}|_{F_{6}})-
      6(|e_y|Q_{b}q_{2y}|_{F_{6}}+|e_x|Q_{b}q_{1x}|_{F_{6}})\right),  \mbox{on } F_{6}.
\end{array}
\right.
\end{equation*}
Let $\widetilde{e_{b}}=(Q_{b}u -u_{b})+h^{2}w_{b}$ be the modified error function. For any $v_{b}\in V_{b}^{0}$, the error function $\widetilde{e_{b}}$ satisfies
\begin{equation}\label{EQ:marker:new}
\begin{split}
&({\color{black}{A}}\nabla_{d}\widetilde{e_{b}},\nabla_{d}v_{b})_h+\rho h^{-1}\sum_{T\in{\cal T}_{h}}\langle Q_{b}{\S}(\widetilde{e_{b}})-\widetilde{e_{b}},Q_{b}{\S}(v_{b})-v_{b}\rangle_{\partial T}     \\
=&h^{2}({\color{black}{A}}\nabla_{d}w_{b},\nabla_{d}v_{b})_h+R_{4}(v_b),
\end{split}
\end{equation}
where $R_{4}(v_b)$ is the remainder satisfying
\begin{eqnarray}\label{a27}
|R_{4}(v_b)|\leq C h^{2} \| u\|_{3} \3bar{\S}(v_{b})-v_{b}\3bar_{\mathcal{E}_h}.
\end{eqnarray}
\end{theorem}
\begin{proof}
It follows from \eqref{a10}, Lemma \ref{lemm1.1} and Lemma \ref{lemma3.3} that
\begin{equation}\label{a28}
\begin{split}
\zeta_{u}(v_{b})=&-\sum_{T\in{\cal T}_{h}}A_{1}\rho h^{-1}|e_x|\chi_{1}(M_{1})\int_{T}u_{xx}dT+\sum_{T\in{\cal T}_{h}}\chi_{1}(M_{1})\int_{T}q_{1x}dT\\
&-\sum_{T\in{\cal T}_{h}}A_{1}\rho h^{-1}|e_y|\chi_{2}(M_{3})\int_{T}u_{yy}dT+\sum_{T\in{\cal T}_{h}}\chi_{2}(M_{3})\int_{T}q_{2y}dT \\
&-\sum_{T\in{\cal T}_{h}}A_{1}\rho h^{-1}|e_z|\chi_{3}(M_{5})\int_{T}u_{zz}dT+\sum_{T\in{\cal T}_{h}}\chi_{3}(M_{5})\int_{T}q_{3z}dT\\
&+\sum_{T\in{\cal T}_{h}}(R_{1}(T)+R_{2}(T)).
\end{split}
\end{equation}

Using the usual integration by parts yields
\begin{eqnarray*}
\int_{T}u_{xx}dT&=&-\int_{T}u_{xxy}(y-y_{c})dT+\frac{1}{2}|e_y|\int_{F_{3}} u_{xx}dF+\frac{1}{2} |e_y|\int_{F_{4}}u_{xx}dF \nonumber \\
&=&-\int_{T}u_{xxy}(y-y_{c})dT+\frac{1}{2}|e_y||F_{3}|Q_{b}u_{xx}|_{F_{3}}+\frac{1}{2}|e_y||F_{4}|Q_{b}u_{xx}|_{F_{4}}.
\end{eqnarray*}

Similarly, we have
{\color{black}{\begin{equation*}
\begin{split}
&\int_{T}u_{yy}dT=-\int_{T}u_{yyz}(z-z_{c})dT+\frac{1}{2}|e_z||F_{5}|Q_{b}u_{yy}|_{F_{5}}+\frac{1}{2}|e_z||F_{6}|Q_{b}u_{yy}|_{F_{6}},\\
&\int_{T}u_{zz}dT=-\int_{T}u_{zzx}(x-x_{c})dT+\frac{1}{2}|e_x||F_{1}|Q_{b}u_{zz}|_{F_{1}}+\frac{1}{2}|e_x||F_{2}|Q_{b}u_{zz}|_{F_{2}}.
\end{split}
\end{equation*}}}

Likewise, there holds
{\color{black}{\begin{equation*}
\begin{split}
&\int_{T}q_{1x}dT=-\int_{T}q_{1xy}(y-y_{c})dT+\frac{1}{2}|e_y||F_{3}|Q_{b}q_{1x}|_{F_{3}}+\frac{1}{2}|e_y||F_{4}|Q_{b}q_{1x}|_{F_{4}},\\
&\int_{T}q_{2y}dT=-\int_{T}q_{2yz}(z-z_{c})dT+\frac{1}{2}|e_z||F_{5}|Q_{b}q_{2y}|_{F_{5}}+\frac{1}{2}|e_z||F_{6}|Q_{b}q_{2y}|_{F_{6}},\\
&\int_{T}q_{3z}dT=-\int_{T}q_{3zx}(x-x_{c})dT+\frac{1}{2}|e_x||F_{1}|Q_{b}q_{3z}|_{F_{1}}+\frac{1}{2}|e_x||F_{2}|Q_{b}q_{3z}|_{F_{2}}.
\end{split}
\end{equation*}}}

Substituting the above identities into \eqref{a28} gives rise to
{\color{black}{\begin{equation}\label{error-6-22:06}
\begin{split}
\zeta_{u}(v_{b})=&-\frac{A_1}{2}\sum_{T\in{\cal T}_{h}}\rho h^{-1}|e_x||e_y||F_3|\chi_{1}(M_{1})(Q_{b}u_{xx}|_{F_{3}}+Q_{b}u_{xx}|_{F_{4}}) \\
&-\frac{A_1}{2}\sum_{T\in{\cal T}_{h}}\rho h^{-1}|e_y||e_z||F_{5}|\chi_{2}(M_{3})(Q_{b}u_{yy}|_{F_{5}}+Q_{b}u_{yy}|_{F_{6}}) \\
&-\frac{A_1}{2}\sum_{T\in{\cal T}_{h}}\rho h^{-1}|e_z||e_x||F_{1}|\chi_{3}(M_{5})(Q_{b}u_{zz}|_{F_{1}}+Q_{b}u_{zz}|_{F_{2}}) \\
&+\frac{1}{2}\sum_{T\in{\cal T}_{h}}|e_y||F_{3}|\chi_{1}(M_{1})(Q_{b}q_{1x}|_{F_{3}}+Q_{b}q_{1x}|_{F_{4}}) \\
&+\frac{1}{2}\sum_{T\in{\cal T}_{h}}|e_z||F_{5}|\chi_{2}(M_{3})(Q_{b}q_{2y}|_{F_{5}}+Q_{b}q_{2y}|_{F_{6}}) \\
&+\frac{1}{2}\sum_{T\in{\cal T}_{h}}|e_x||F_{1}|\chi_{3}(M_{5})(Q_{b}q_{3z}|_{F_{1}}+Q_{b}q_{3z}|_{F_{2}})+\sum_{T\in{\cal T}_{h}}R_{3}(T),
\end{split}
\end{equation}}}
where the remainder term $R_{3}(T)$ is given by
{\color{black}{\begin{equation*}\label{a29}
\begin{split}
R_{3}(T)=&A_{1}\rho h^{-1}|e_x|\chi_{1}(M_{1})\int_{T}u_{xxy}(y-y_{c})dT-\chi_{1}(M_{1})\int_{T}q_{1xy}(y-y_{c})dT\\
&+A_{1}\rho h^{-1}|e_y|\chi_{2}(M_{3})\int_{T}u_{yyz}(z-z_{c})dT-\chi_{2}(M_{3})\int_{T}q_{2yz}(z-z_{c})dT\\
&+A_{1}\rho h^{-1}|e_z|\chi_{3}(M_{5})\int_{T}u_{zzx}(x-x_{c})dT-\chi_{3}(M_{5})\int_{T}q_{3zx}(x-x_{c})dT\\
&+R_{1}(T)+R_{2}(T).
\end{split}
\end{equation*}}}

From the definition of $Q_{b}$ and the usual integration by parts, we arrive at
\begin{equation}\label{eq1}
\begin{split}
&\frac{1}{2}|e_x|^{2}|F_{5}|\chi_{3}(M_{5})(Q_{b}u_{xx}|_{F_{3}}+Q_{b}u_{xx}|_{F_{4}}) \\
=&\frac{|e_x|^{2}|F_5|}{|e_y||F_3|}\chi_{3}(M_{5})(\int_{T}u_{xx}dT+\int_{T}u_{xxy}(y-y_{c})dT) \\
=&\frac{|e_x|^{2}|F_5|}{2}\chi_{3}(M_{5})(Q_{b}u_{xx}|_{F_{5}}+Q_{b}u_{xx}|_{F_{6}}) \\
&+\frac{|e_x|^{2}}{|e_z|}\chi_{3}(M_{5})(\int_{T}u_{xxy}(y-y_{c})dT-\int_{T}u_{xxz}(z-z_{c})dT).
\end{split}
\end{equation}
Similarly, we have
\begin{equation}\label{eq2}
\begin{split}
&\frac{1}{2}|e_x||F_{5}|\chi_{3}(M_{5})(Q_{b}q_{1x}|_{F_{3}}+Q_{b}q_{1x}|_{F_{4}}) \\
=&\frac{|e_x||F_5|}{|e_y||F_3|}\chi_{3}(M_{5})(\int_{T}q_{1x}dT+\int_{T}q_{1xy}(y-y_{c})dT) \\
=&\frac{|e_x||F_5|}{2}\chi_{3}(M_{5})(Q_{b}q_{1x}|_{F_{5}}+Q_{b}q_{1x}|_{F_{6}}) \\
&+\frac{|e_x|}{|e_z|}\chi_{3}(M_{5})(\int_{T}q_{1xy}(y-y_{c})dT-\int_{T}q_{1xz}(z-z_{c})dT),
\end{split}
\end{equation}
 \begin{equation}\label{eq3}
\begin{split}
&\frac{1}{2}|e_y|^{2}|F_{1}|\chi_{1}(M_{1})(Q_{b}u_{yy}|_{F_{5}}+Q_{b}u_{yy}|_{F_{6}}) \\
=&\frac{|e_y|^{2}|F_1|}{|e_z||F_5|}\chi_{1}(M_{1})(\int_{T}u_{yy}dT+\int_{T}u_{yyz}(z-z_{c})dT) \\
=&\frac{|e_y|^{2}|F_1|}{2}\chi_{1}(M_{1})(Q_{b}u_{yy}|_{F_{1}}+Q_{b}u_{yy}|_{F_{2}}) \\
&+\frac{|e_y|^{2}}{|e_x|}\chi_{1}(M_{1})(\int_{T}u_{yyz}(z-z_{c})dT-\int_{T}u_{yyx}(x-x_{c})dT),
\end{split}
\end{equation}
 \begin{equation}\label{eq4}
\begin{split}
&\frac{1}{2}|e_z|^{2}|F_{3}|\chi_{2}(M_{3})(Q_{b}u_{zz}|_{F_{1}}+Q_{b}u_{zz}|_{F_{2}}) \\
=&\frac{|e_z|^{2}|F_3|}{|e_x||F_1|}\chi_{2}(M_{3})(\int_{T}u_{zz}dT+\int_{T}u_{zzx}(x-x_{c})dT) \\
=&\frac{|e_z|^{2}|F_3|}{2}\chi_{2}(M_{3})(Q_{b}u_{zz}|_{F_{3}}+Q_{b}u_{zz}|_{F_{4}}) \\
&+\frac{|e_z|^{2}}{|e_y|}\chi_{2}(M_{3})(\int_{T}u_{zzx}(x-x_{c})dT-\int_{T}u_{zzy}(y-y_{c})dT),
\end{split}
\end{equation}
\begin{equation}\label{eq5}
\begin{split}
&\frac{1}{2}|e_y||F_{1}|\chi_{1}(M_{1})(Q_{b}q_{2y}|_{F_{5}}+Q_{b}q_{2y}|_{F_{6}}) \\
=&\frac{|e_y||F_1|}{2}\chi_{1}(M_{1})(Q_{b}q_{2y}|_{F_{1}}+Q_{b}q_{2y}|_{F_{2}}) \\
&+\frac{|e_y|}{|e_x|}\chi_{1}(M_{1})(\int_{T}q_{2yz}(z-z_{c})dT-\int_{T}q_{2yx}(x-x_{c})dT),
\end{split}
\end{equation}
 \begin{equation}\label{eq6}
\begin{split}
&\frac{1}{2}|e_z||F_{3}|\chi_{2}(M_{3})(Q_{b}q_{3z}|_{F_{1}}+Q_{b}q_{3z}|_{F_{2}}) \\
=&\frac{|e_z||F_3|}{2}\chi_{2}(M_{3})(Q_{b}q_{3z}|_{F_{3}}+Q_{b}q_{3z}|_{F_{4}}) \\
&+\frac{|e_z|}{|e_y|}\chi_{2}(M_{3})(\int_{T}q_{3zx}(x-x_{c})dT-\int_{T}q_{3zy}(y-y_{c})dT).
\end{split}
\end{equation}

Using \eqref{FPS} and (\ref{eq1})- (\ref{eq6}), \eqref{error-6-22:06} can be rewritten as
 \begin{equation}\label{error-6-22:09}
\begin{split}
&\zeta_{u}(v_{b})\\=&\frac{A_1}{2}\sum_{T\in\T_h}\rho h^{-1}|F_1|\chi_{1}(M_1)\Big (|e_y|^{2}(Q_{b}u_{yy}|_{F_{1}}+Q_{b}u_{yy}|_{F_{2}})\\&+
|e_z|^{2}(Q_{b}u_{zz}|_{F_{1}}+Q_{b}u_{zz}|_{F_{2}})\Big) \\
&+\frac{A_1}{2}\sum_{T\in\T_h}\rho h^{-1}|F_3|\chi_{2}(M_3)\Big(|e_x|^{2}(Q_{b}u_{xx}|_{F_{3}}+Q_{b}u_{xx}|_{F_{4}})\\&+
|e_z|^{2}(Q_{b}u_{zz}|_{F_{3}}+Q_{b}u_{zz}|_{F_{4}})\Big) \\
&+\frac{A_1}{2}\sum_{T\in\T_h}\rho h^{-1}|F_5|\chi_{3}(M_5)\Big(|e_y|^{2}(Q_{b}u_{yy}|_{F_{5}}+Q_{b}u_{yy}|_{F_{6}})\\&+
|e_x|^{2}(Q_{b}u_{xx}|_{F_{5}}+Q_{b}u_{xx}|_{F_{6}})\Big) \\
&-\frac{1}{2}\sum_{T\in\T_h}|F_1|\chi_{1}(M_1)\Big(|e_y|(Q_{b}q_{2y}|_{F_{1}}+Q_{b}q_{2y}|_{F_{2}})\\&+
|e_z|(Q_{b}q_{3z}|_{F_{1}}+Q_{b}q_{3z}|_{F_{2}})\Big) \\
&-\frac{1}{2}\sum_{T\in\T_h}|F_3|\chi_{2}(M_3)\Big(|e_x|(Q_{b}q_{1x}|_{F_{3}}+Q_{b}q_{1x}|_{F_{4}})\\&+
|e_z|(Q_{b}q_{3z}|_{F_{3}}+Q_{b}q_{3z}|_{F_{4}})\Big) \\
&-\frac{1}{2}\sum_{T\in\T_h}|F_5|\chi_{3}(M_5)\Big (|e_x|(Q_{b}q_{1x}|_{F_{5}}+Q_{b}q_{1x}|_{F_{6}})\\&+
 |e_y| (Q_{b}q_{2y}|_{F_{5}}+Q_{b}q_{2y}|_{F_{6}})\Big) +\sum_{T\in\T_h}R_4(T),
\end{split}
\end{equation}
where the remainder term $R_{4}(T)$ is given by
\begin{equation*}
\begin{split}
R_{4}(T)=&A_{1}\frac{\rho h^{-1}|e_x|^{2}}{|e_z|}\chi_{3}(M_{5})(\int_{T}u_{xxy}(y-y_{c})dT-\int_{T}u_{xxz}(z-z_{c})dT)\\
&+A_{1}\frac{\rho h^{-1}|e_y|^{2}}{|e_x|}\chi_{1}(M_{1})(\int_{T}u_{yyz}(z-z_{c})dT-\int_{T}u_{yyx}(x-x_{c})dT)\\
&+A_{1}\frac{\rho h^{-1}|e_z|^{2}}{|e_y|}\chi_{2}(M_{3})(\int_{T}u_{zzx}(x-x_{c})dT-\int_{T}u_{zzy}(y-y_{c})dT)\\
&-\frac{|e_x|}{|e_z|}\chi_{3}(M_{5})(\int_{T}q_{1xy}(y-y_{c})dT-\int_{T}q_{1xz}(z-z_{c})dT)\\
&-\frac{|e_y|}{|e_x|}\chi_{1}(M_{1})(\int_{T}q_{2yz}(z-z_{c})dT-\int_{T}q_{2yx}(x-x_{c})dT)\\
&-\frac{|e_z|}{|e_y|}\chi_{2}(M_{3})(\int_{T}q_{3zx}(x-x_{c})dT-\int_{T}q_{3zy}(y-y_{c})dT)\\
&+R_{3}(T).
\end{split}
\end{equation*}

Letting
 \begin{equation*}\label{EQ:wwb}
w_{b}=\left\{
\begin{array}{l}
      \frac{1}{2}\rho^{-1}h^{-1}\left(\rho h^{-1}A_{1}(|e_y|^{2}Q_{b}u_{yy}|_{F_{1}}+|e_z|^{2}Q_{b}u_{zz}|_{F_{1}})-
      |e_y|Q_{b}q_{2y}|_{F_{1}}-|e_z|Q_{b}q_{3z}|_{F_{1}}\right), \mbox{on } F_{1},\\
     \frac{1}{2}\rho^{-1}h^{-1}\left(\rho h^{-1}A_{1}(|e_y|^{2}Q_{b}u_{yy}|_{F_{2}}+|e_z|^{2}Q_{b}u_{zz}|_{F_{2}})-
      |e_y|Q_{b}q_{2y}|_{F_{2}}-|e_z|Q_{b}q_{3z}|_{F_{2}}\right), \mbox{on } F_{2},\\
      \frac{1}{2}\rho^{-1}h^{-1}\left(\rho h^{-1}A_{1}(|e_x|^{2}Q_{b}u_{xx}|_{F_{3}}+|e_z|^{2}Q_{b}u_{zz}|_{F_{3}})-
      |e_x|Q_{b}q_{1x}|_{F_{3}}-|e_z|Q_{b}q_{3z}|_{F_{3}}\right), \mbox{on } F_{3},\\
     \frac{1}{2}\rho^{-1}h^{-1}\left(\rho h^{-1}A_{1}(|e_x|^{2}Q_{b}u_{xx}|_{F_{4}}+|e_z|^{2}Q_{b}u_{zz}|_{F_{4}})-
      |e_x|Q_{b}q_{1x}|_{F_{4}}-|e_z|Q_{b}q_{3z}|_{F_{4}}\right), \mbox{on } F_{4},\\
      \frac{1}{2}\rho^{-1}h^{-1}\left(\rho h^{-1}A_{1}(|e_y|^{2}Q_{b}u_{yy}|_{F_{5}}+|e_x|^{2}Q_{b}u_{xx}|_{F_{5}})-
      |e_x|Q_{b}q_{1x}|_{F_{5}}-|e_y|Q_{b}q_{2y}|_{F_{5}}\right),  \mbox{on } F_{5},\\
       \frac{1}{2}\rho^{-1}h^{-1}\left(\rho h^{-1}A_{1}(|e_y|^{2}Q_{b}u_{yy}|_{F_{6}}+|e_x|^{2}Q_{b}u_{xx}|_{F_{6}})-
      |e_x|Q_{b}q_{1x}|_{F_{6}}-|e_y|Q_{b}q_{2y}|_{F_{6}}\right), \mbox{on } F_{6}.
\end{array}
\right.
\end{equation*}
Thus, \eqref{error-6-22:09} is rewritten as
 \begin{eqnarray}\label{sperconve-error-estimate-6-6}
\begin{split}
\zeta_{u}(v_{b})=&\rho h {\sum_{T\in\T_h}} \langle w_{b},Q_{b}{\S}(v_{b})-v_{b}\rangle_{\partial T}+{\color{black}{\sum_{T\in\T_h}}}R_{4}(T)\\
=&-\rho h\sum_{T\in{\cal T}_{h}}\langle Q_{b}{\S}(w_{b})-w_{b},Q_{b}{\S}(v_{b})-v_{b}\rangle_{\partial T}+\sum_{T\in{\cal T}_{h}}R_{4}(T),
\end{split}
\end{eqnarray}
where we use (\ref{a7}) on the last line.

Substituting \eqref{sperconve-error-estimate-6-6} into  \eqref{a9} we obtain
\begin{eqnarray}\label{sperconve-error-equation-10-15}
\begin{split}
&({\color{black}A}\nabla_{d}e_{b},\nabla_{d}v_{b})_{h}+\rho h^{-1}\sum_{T\in{\cal T}_{h}}\langle Q_{b}{\S}(e_{b})-e_{b},Q_{b}{\S}(v_{b})-v_{b}\rangle_{\partial T}     \\
=&-\rho h\sum_{T\in{\cal T}_{h}}\langle Q_{b}{\S}(w_{b})-w_{b},Q_{b}{\S}(v_{b})-v_{b}\rangle_{\partial T}+R_{4}(v_b),
\end{split}
\end{eqnarray}
where $R_{4}(v_b)=\sum_{T\in{\cal T}_{h}}R_{4}(T)$.

Letting $\widetilde{e_{b}}=e_{b}+h^{2}w_{b}$, we arrive at
\begin{eqnarray*}
\begin{split}
&({\color{black}A}\nabla_{d}\widetilde{e_{b}},\nabla_{d}v_{b})_{h}+\rho h^{-1}\sum_{T\in{\cal T}_{h}}\langle Q_{b}{\S}(\widetilde{e_{b}})-\widetilde{e_{b}}, Q_{b}{\S}(v_{b})-v_{b}\rangle_{\partial T}     \\
=&h^{2}({\color{black}A}\nabla_{d}w_{b},\nabla_{d}v_{b})_{h}+R_{4}(v_b).
\end{split}
\end{eqnarray*}

It is easy to see from the definition of $w_{b}$  that
 \begin{eqnarray}\label{sperconve-pertur-estimate-1-1-18:37}
\|\nabla_d w_b\|_0 \leq C \|u\|_3,
\end{eqnarray}
from which, \eqref{a27} is obtained in a similar way of the proof of \eqref{RE-1-1}. This completes the proof of the theorem.
\end{proof}

For the second order elliptic problem \eqref{model} in three dimensions with the homogeneous Dirichlet boundary value, we have $w_b \in V_{b}^{0}$. Letting $v_{b}=\widetilde{e_{b}}$ in \eqref{EQ:marker:new} gives a superconvergence estimate
\begin{eqnarray*}
\Big({\color{black}\sum_{T\in{\cal T}_{h}}}\|\nabla_{d}\widetilde{e_{b}} \|_{T}^{2}\Big)^{\frac{1}{2}}\leq C h^{2}\| u \|_{3}.
\end{eqnarray*}
For the second order elliptic problem \eqref{model} in three dimensions with the nonhomogeneous Dirichlet boundary condition, we have   $w_b \notin V_{b}^{0}$. Thus, $\widetilde{e_{b}}=(Q_bu-u_b)+h^2 w_b \notin V_{b}^{0}$. In order to obtain a superconvergence estimate, we enforce a computational solution $u_b$ satisfying the following condition:
\begin{eqnarray}\label{modified-boundary-date}
u_{b}=Q_{b}g+h^{2}w_{b}, \quad \text{on} \  \partial\Omega.
\end{eqnarray}
The above boundary condition is able to be implemented if $w_{b}|_{\partial \Omega}$ is computable without any prior knowledge of the exact solution $u$. The following theorem and corollary  assume that $w_{b}|_{\partial \Omega}$ is computable.

\begin{theorem}\label{superconvergence-H-1-2}
Assume that $u \in H^{3}(\Omega)$ is the exact solution of the second order elliptic model problem \eqref{model} in three dimensions and $u_{b}\in V_{b}^{g}$ is the numerical solution of the simplified weak Galerkin finite element scheme (\ref{SWG}). Let $w_{b}\in V_{b}$ be a given function defined in Theorem \ref{superconvergence-H-1-1}. Denote by $\widetilde{e_{b}}=(Q_{b}u-u_{b})+h^{2}w_{b}$ the modified error function. The following superconvergence estimate holds true:
\begin{equation}\label{EQ:SuperC:17:16}
\left(\sum_{T\in\T_{h}}\|\nabla_{d}\widetilde{e_{b}} \|_{T}^{2}\right)^{\frac12}+ \3bar{\S}(\widetilde{e_{b}})-\widetilde{e_{b}}\3bar_{\mathcal{E}_h} \leq Ch^{2}\|u\|_{3}.
\end{equation}
\end{theorem}
\begin{proof}
Letting $v_{b}=\widetilde{e_{b}}\in V_{b}^{0}$ in \eqref{EQ:marker:new}  gives
\begin{eqnarray*}
\begin{split}
&({\color{black}A}\nabla_{d}\widetilde{e_{b}},\nabla_{d}\widetilde{e_{b}})_h+\rho h^{-1}\sum_{T\in{\cal T}_{h}}\langle Q_{b}{\S}(\widetilde{e_{b}})-\widetilde{e_{b}},Q_{b}{\S}(\widetilde{e_{b}})-\widetilde{e_{b}}\rangle_{\partial T}     \\
=&h^{2}({\color{black}A}\nabla_{d}w_{b},\nabla_{d}\widetilde{e_{b}})_h+R_{4}(\widetilde{e_{b}}).
\end{split}
\end{eqnarray*}

It follows from \eqref{a27}, \eqref{sperconve-pertur-estimate-1-1-18:37}, \eqref{RE-2018-6-3} and the Cauchy-Schwarz inequality that
\begin{eqnarray*}
&&({\color{black}A}\nabla_{d}\widetilde{e_{b}},\nabla_{d}\widetilde{e_{b}})_h+\rho h^{-1}\sum_{T\in{\cal T}_{h}}\|Q_{b}{\S}(\widetilde{e_{b}})-\widetilde{e_{b}} \|_{\partial T}^{2}\\
&\leq& C h^{2}\| u\|_{3}(\|\nabla_{d}\widetilde{e_{b}}\|_0+\3bar{\S}(\widetilde{e_{b}})-\widetilde{e_{b}}\3bar_{\mathcal{E}_h})\\
&\leq& C h^{4} \| u\|_{3}^{2}+ \frac{\|A\|_{\infty}}{2} \|\nabla_{d}\widetilde{e_{b}}\|^2_0+ \frac{1}{2}\3bar{\S}(\widetilde{e_{b}})-\widetilde{e_{b}}\3bar_{\mathcal{E}_h}^2,
\end{eqnarray*}
which leads to
$$
\sum_{T\in{\cal T}_{h}}\|\nabla_d  \tilde{e}_b\|_{T}^2 + \3bar {\S}(\tilde{e}_b)-\tilde{e}_b\3bar^2_{\mathcal{E}_h} \leq  C h^4 \|u\|_3^2.
$$
This completes the proof of the theorem.
\end{proof}

We can see from \eqref{modified-boundary-date} that the standard $L^{2}$ projection of Dirichlet data $g$ is perturbed by
\begin{eqnarray*}\label{EQ:PL2:19:01}
\varepsilon_b:=\frac{1}{12}\left(|e_y|^{2}Q_{b}g_{yy}+|e_z|^{2}Q_{b}g_{zz}-6\rho^{-1}h(|e_y|Q_{b}q_{2y}+|e_z|Q_{b}q_{3z})\right)
\end{eqnarray*}
on boundary faces parallel to the flat face $F_1$ (or $F_2$),
\begin{eqnarray*}\label{EQ:PL2:19:02}
\varepsilon_b:=\frac{1}{12}\left(|e_x|^{2}Q_{b}g_{xx}+|e_z|^{2}Q_{b}g_{zz}-6\rho^{-1}h(|e_x|Q_{b}q_{1x}+|e_z|Q_{b}q_{3z})\right)
\end{eqnarray*}
on boundary faces parallel to the flat face $F_3$ (or $F_4$), and
\begin{eqnarray*}\label{EQ:PL2:19:03}
\varepsilon_b:=\frac{1}{12}\left(|e_x|^{2}Q_{b}g_{xx}+|e_y|^{2}Q_{b}g_{yy}-6\rho^{-1}h(|e_x|Q_{b}q_{1x}+|e_y|Q_{b}q_{2y})\right)
\end{eqnarray*}
on boundary faces parallel to the flat face $F_5$ (or $F_6$). For the Dirichlet boundary value problem with a diagonal diffusive tensor $A=(a_{11}, 0, 0; 0, a_{22}, 0; 0, 0, a_{33})$, the perturbation $\varepsilon_b$ is computable using the boundary data $g$ and thus the mixed partial derivatives of $u$ are not needed. The following superconvergence estimate is particularly for the second order elliptic problem in three dimensions with a diagonal diffusive tensor $A$.

\begin{corollary}\label{Corollary:SuperC-nonhomo}
Assume that $u\in H^{3}(\Omega)$ is the exact solution of the model problem (\ref{model}) in three dimensions with a diagonal diffusive tensor $A=(a_{11}, 0, 0; 0, a_{22}, 0; 0, 0, a_{33})$. Let $u_{b}\in V_{b}$ be the weak Galerkin finite element solution arising from the scheme (\ref{SWG}) with the boundary values specified as follows: on the boundary faces which are parallel to the flat face $F_1$ (or $F_2$), let
\begin{eqnarray}\label{EQ:PL2:01}
u_{b}=Q_{b}g+\frac{1}{12}\left(|e_y|(|e_y|-\frac{6}{\rho} ha_{22})Q_{b}g_{yy}+|e_z|(|e_z|-\frac{6}{\rho}ha_{33})Q_{b}g_{zz}\right)
\end{eqnarray}
on the boundary faces which are parallel to the flat face $F_3$ (or $F_4$), let
\begin{eqnarray}\label{EQ:PL2:02}
u_{b}=Q_{b}g+\frac{1}{12}\left(|e_x|(|e_x|-\frac{6}{\rho}ha_{11})Q_{b}g_{xx}+|e_z|(|e_z|-\frac{6}{\rho}ha_{33})Q_{b}g_{zz}\right)
\end{eqnarray}
and on the boundary faces which are parallel to the flat face $F_5$ (or $F_6$), let
\begin{eqnarray}\label{EQ:PL2:03}
u_{b}=Q_{b}g+\frac{1}{12}\left(|e_y|(|e_y|-\frac{6}{\rho}ha_{22})Q_{b}g_{yy}+|e_x|(|e_x|-\frac{6}{\rho}ha_{11})Q_{b}g_{xx}\right).
\end{eqnarray}
There holds
\begin{eqnarray}\label{sperconve-coro-estimate-6-6}
\left(\sum_{T\in{\cal T}_{h}}\|\mathbb{Q}_h\nabla u-\nabla_{d}u_{b} \|_{T}^{2}\right)^{1/2}\leq Ch^{2}\|u\|_{3}.
\end{eqnarray}
\end{corollary}
\begin{proof}
Note that the perturbation $\varepsilon_b$ is of the order ${\cal{O}}(h^2)$. It follows from Theorem \ref{superconvergence-H-1-2} that
\begin{eqnarray*}
\left(\sum_{T\in{\cal T}_{h}}\| \nabla_{d}Q_{b}u-\nabla_{d}u_{b}\|_{T}^{2}\right)^{1/2}\leq Ch^{2}\| u\|_{3},
\end{eqnarray*}
which combined with the commutative property (\ref{COMM-PRO}), yields
\begin{eqnarray*}
\left(\sum_{T\in{\cal T}_{h}}\|\mathbb{Q}_h\nabla u-\nabla_{d}u_{b} \|_{T}^{2}\right)^{1/2}\leq Ch^{2}\|u\|_{3}.
\end{eqnarray*}
This completes the proof of the corollary.
\end{proof}

For the model problem \eqref{model}  in three dimensions with arbitrary diffusive coefficient $A$, the following superconvergence error estimate holds true.

\begin{theorem}\label{superconvergence-H-1-53}
Assume that $u\in H^{3}(\Omega)$ is the exact solution of the second order elliptic model problem (\ref{model}) in three dimensions. Let $u_{b}\in V_{b}$ be the weak Galerkin finite element solution arising from the scheme \eqref{SWG} with the boundary value $u_{b}=Q_{b}g$ on $\partial\Omega$. There holds\begin{equation*}\label{EQ:SuperC:053-generic}
\left(\sum_{T\in\T_{h}}\|\nabla_{d}{e_{b}} \|_{T}^{2}\right)^{\frac12} \leq Ch^{1.5}(\|u\|_{3}+\|\nabla^2 u\|_{0,\partial\Omega}).
\end{equation*}
\end{theorem}
\begin{proof}
The proof is the similar to the proof of Theorem 6.7 in \cite{sup_LWW2018}, and therefore the details are omitted here.
\end{proof}

\begin{corollary}\label{newsuperconvergence-H-1-10-10}
Let $u_{b}\in V_{b}$ such that  $u_b= Q_b g$ on $\partial\Omega$ be the weak Galerkin finite element solution of the model problem \eqref{model} in three dimensions arising from the scheme \eqref{SWG}. Assume that (1) the exact solution $u\in H^{3}(\Omega)$ of the model problem (\ref{model}) satisfies $u_{xx}=u_{yy}=u_{zz}$; (2) $\T_{h}$ is a uniform cubic partition $\O$ with $|e_x|=|e_y|=|e_z|$; and (3) the diffusive tensor $A=(a_{11}, 0, 0; 0, a_{22}, 0; 0, 0, a_{33})$ satisfies $a_{11}=a_{22}=a_{33}$. Denote by $e_b=Q_bu-u_b$. The following superconvergence result holds true:
\begin{equation*}\label{EQ:new-SuperC:053-generic}
\left(\sum_{T\in\T_{h}}\|\nabla_{d}{e_{b}} \|_{T}^{2}\right)^{\frac12} \leq Ch^{2}\|u\|_{3}.
\end{equation*}
\end{corollary}
\begin{proof}
Using Theorem \ref{superconvergence-H-1-1}, Lemma \ref{lemma2.2} and the error equation \eqref{sperconve-error-equation-10-15} concludes the corollary. Details are omitted here due to page limitation.
\end{proof}

\section{Numerical Experiments}\label{Section:NE}
In this section, a series of numerical tests will be demonstrated for the simplified WG algorithm \eqref{SWG} for solving the second order elliptic problem \eqref{model} in three dimensions to verify the superconvergence error estimates established in the precious sections.

The numerical tests are based on the lowest order (i.e., $k=1$) of weak functions on the uniform and non-uniform cubic partitions of the unit cube $\O=(0, 1)^3$. More precisely, the local weak finite element space is given by $V(T, 1)=\{v=\{v_{0},v_{b}\},v_{0} \in P_{1}(T),v_{b}  \in P_{0}(F), F\subset \partial T\}$, and $\nabla_{d}v|_T\in[P_{0}(T)]^{3}$.

Let $u$ be the exact solution of the model problem (\ref{model}). The error function $e_h$ is given by $e_h=Q_hu-u_h=\{e_0, e_b\}$ where $e_{0}=Q_{0}u-{\S}(u_{b})$ and $e_{b}=Q_{b}u-u_{b}$. The error functions are measured in various norms as follows:
\begin{equation*}
\begin{split}
&Discrete\ L^{\infty}- norm:\quad\|u-{\S}(u_{b})\|_{\infty,\star}=\max\limits_{T\in {\cal T}_h}\big|(u-{\S}(u_{b}))(M_{c})\big|,\\
&L^{2}- norm:\quad\|e_0\|_0= \Big(\sum_{T\in{\cal T}_{h}}\int_T|Q_{0}u-{\S}(u_{b})|^2dT\Big)^{1/2},               \\
&H^{1}- norm:\quad\|\nabla_d e_{b}\|_0=\Big(\sum_{T\in{\cal T}_{h}}\int_T|\nabla_d (Q_{b}u-u_{b})|^2dT\Big)^{1/2},\\
&Discrete\ W^{1, 1}- norm:\quad\|\nabla_d u_b-\nabla u\|_{1,\star}=\Big(\sum_{T\in{\cal T}_{h}}\int_T|\nabla_d u_b-\nabla u(M_{c})|^{2}dT\Big)^{1/2},\\
&W^{1,1}- seminorm:\quad\|e_{0}\|_{1,1}=\Big(\sum_{T\in{\cal T}_{h}}\int_T|\nabla (Q_{0}u- {\S}(u_{b}))|^{2}dT\Big)^{1/2},
\end{split}
\end{equation*}
where $M_{c}$ is the center of the cubic element $T$.

\subsection{Numerical experiments for constant diffusion tensor $A$}\label{subSection:NE1}

{\bf Test Case 1 (Homogeneous BVP)} In this set of tests,  the diffusive coefficient tensor $A$ is an identity matrix and the exact solution is $u =\sin(\pi x)\sin(\pi y)\sin(\pi z)$. This is a homogeneous boundary value problem on the domain $\O=(0, 1)^3$.

Tables \ref{test case 1 :rho6:UniformWP} - \ref{test case 1 :rho1:non-UniformWP} illustrate the numerical results on the uniform and non-uniform cubic partitions with the stabilization parameters $\rho=6$ and $\rho=1$, respectively. These results show that the convergence rate for the error $e_b$ in $H^{1}$-norm is of order ${\cal{O}}(h^2)$, which is consistent with the conclusion in Corollary \ref{Corollary:SuperC-nonhomo}. We also compute the convergence rates for the error functions in $||\cdot||_{1,\star}$ norm and $\|\cdot\|_{1,1}$ norm which seem to be in the superconvergence order of ${\cal{O}}(h^2)$, although there are not any corresponding theories available in this paper.

It is interesting to see from Table \ref{test case 1 :rho6:UniformWP} and Table \ref{test case 1 :rho1:UniformWP} that the numerical results corresponding to  stabilization parameters $\rho=1$ and $\rho=6$ are very close to each other.   Furthermore, we compute the numerical results for some other stabilization parameter $\rho = 0.01, 0.1, 2, 5$, and we found the numerical results are still very close to the results in Table \ref{test case 1 :rho6:UniformWP}. Due to page limitation, we shall not demonstrate the  numerical results for $\rho = 0.01, 0.1, 2, 5$ in this paper. Interested readers are welcome to draw their own conclusions for this phenomenon.

%

\begin{table}[htbp]\centering\scriptsize
\caption{Test Case 1: Convergence of the lowest order WG-FEM on the unit cubic domain with exact solution $u=\sin(\pi x)\sin(\pi y)\sin(\pi z)$, uniform cubic partitions, stabilization parameter $\rho=6$.}\label{test case 1 :rho6:UniformWP}
{
\setlength{\extrarowheight}{1.5pt}
\begin{center}
\begin{tabular}{|l|l|l|l|l|l|}
\hline $meshes$ &$\|u-\S(u_{b})\|_{\infty,\star}$& $\|e_0\|_0$ & $\|\nabla_d{e_{b}}\|_0$  &$\|\nabla_d u_b-\nabla u\|_{1,\star}$  &    $\|e_0\|_{1,1}$ \\
\hline
$4\times4\times4$     &2.4845e-02    &1.9393e-02    &1.8494e-01               &4.1467e-02     &1.6637e-01\\     \hline
$8\times8\times8$     &6.4194e-03    &4.6306e-03     &4.8626e-02                &1.1850e-02    &4.3758e-02 \\      \hline
$16\times16\times16$   &1.6069e-03    &1.1415e-03     &1.2310e-02                &3.0582e-03   &1.1079e-02   \\   \hline
$32\times32\times32$   &4.0164e-04   &2.8433e-04      &3.0872e-03          &7.7058e-04        &2.7784e-03\\   \hline
\hline
Rate        &2.00  &2.01  &2.00 & 1.99  &2.00 \\
\hline
\end{tabular}
\end{center}
}
\end{table}

\begin{table}[htbp]\centering\scriptsize
\caption{Test Case 1: Convergence of the lowest order WG-FEM on the unit cubic domain with exact solution $u=\sin(\pi x)\sin(\pi y)\sin(\pi z)$, non-uniform cubic partitions, stabilization parameter $\rho=6$, and $h=(|e_x|^2+|e_y|^2+|e_z|^2)^{\frac{1}{2}}$.}.\label{test case 2 :rho6:non-UniformWP}
{
\setlength{\extrarowheight}{1.5pt}
\begin{center}
\begin{tabular}{|l|l|l|l|l|l|}
\hline $meshes$ &$\|u-\S(u_{b})\|_{\infty,\star}$& $\|e_0\|_0$ & $\|\nabla_d{e_{b}}\|_0$  &$\|\nabla_d u_b-\nabla u\|_{1,\star}$  &    $\|e_0\|_{1,1}$ \\
\hline
$3\times4\times5$     &3.0558e-02    &2.3666e-02    &2.1210e-01               &5.2931e-02     &1.9037e-01\\     \hline
$6\times8\times10$     &6.3404e-03    &5.6370e-03     &5.5494e-02                &1.3847e-02    &4.9893e-02 \\      \hline
$12\times16\times20$   &1.5721e-03    &1.3928e-03     &1.4036e-02                &3.5264e-03   &1.2625e-02   \\       \hline
$24\times32\times40$   &3.9192e-04    &3.4718e-04     &3.5192e-03            &8.8605e-04     &3.1660e-03\\           \hline
\hline
Rate        &2.00          &2.00      &2.00                      & 1.99  &2.00 \\
\hline
\end{tabular}
\end{center}
}
\end{table}

%
%

\begin{table}[htbp]\centering\scriptsize
\caption{Test Case 1: Convergence of the lowest order WG-FEM on the unit cubic domain with exact solution $u=\sin(\pi x)\sin(\pi y)\sin(\pi z)$, uniform cubic partitions, stabilization parameter $\rho=1
$.}\label{test case 1 :rho1:UniformWP}
{
\setlength{\extrarowheight}{1.5pt}
\begin{center}
\begin{tabular}{|l|l|l|l|l|l|}
\hline $meshes$ &$\|u-\S(u_{b})\|_{\infty,\star}$& $\|e_0\|_0$ & $\|\nabla_d{e_{b}}\|_0$  &$\|\nabla_d u_b-\nabla u\|_{1,\star}$  &    $\|e_0\|_{1,1}$ \\
\hline
$4\times4\times4$     &2.4845e-02    &1.9393e-02    &1.8494e-01               &4.1467e-02     &1.6637e-01\\     \hline
$8\times8\times8$     &6.4194e-03    &4.6306e-03     &4.8626e-02                &1.1850e-02    &4.3758e-02 \\      \hline
$16\times16\times16$   &1.6069e-03    &1.1415e-03     &1.2310e-02                &3.0582e-03   &1.1079e-02   \\       \hline
$32\times32\times32$   &4.0164e-04   &2.8433e-04      &3.0872e-03          &7.7058e-04        &2.7784e-03\\           \hline
\hline
Rate        &2.00          &2.01      &2.00                      & 1.99  &2.00 \\
\hline
\end{tabular}
\end{center}
}
\end{table}

\begin{table}[htbp]\centering\scriptsize
\caption{Test Case 1: Convergence of the lowest order WG-FEM on the unit cubic domain with exact solution $u=\sin(\pi x)\sin(\pi y)\sin(\pi z)$, non-uniform cubic partitions, stabilization parameter $\rho=1$, and $h=\max(|e_x|,|e_y|,|e_z|)$.}\label{test case 1 :rho1:non-UniformWP}
{
\setlength{\extrarowheight}{1.5pt}
\begin{center}
\begin{tabular}{|l|l|l|l|l|l|}
\hline $meshes$ &$\|u-\S(u_{b})\|_{\infty,\star}$& $\|e_0\|_0$ & $\|\nabla_d{e_{b}}\|_0$  &$\|\nabla_d u_b-\nabla u\|_{1,\star}$  &    $\|e_0\|_{1,1}$ \\
\hline
$3\times4\times5$     &2.2605e-02    &2.5271e-02    &2.1817e-01               &6.9417e-02     &1.9998e-01\\     \hline
$6\times8\times10$     &3.4472e-03    &6.6425e-03     &6.1177e-02                &2.8983e-02    &5.7273e-02 \\      \hline
$12\times16\times20$   &7.3558e-04    &1.6886e-03     &1.5931e-02                &8.3053e-03   &1.5017e-02   \\       \hline
$24\times32\times40$   &1.7500e-04    &4.2391e-04   &4.0277e-03                  &2.1491e-03   &3.8034e-03\\           \hline
\hline
Rate        &2.07          &1.99      &1.98  & 1.95  &1.98 \\
\hline
\end{tabular}
\end{center}
}
\end{table}

{\bf Test Case 2 (Nonhomogeneous BVP)} In this set of tests, the exact solution is $u=\cos(x)\sin(y)\cos(z)$ and the coefficient matrix $A$ is an identity matrix. This is a non-homogeneous boundary value problem on the domain $\O=(0, 1)^3$.

In Tables \ref{test case 2 :rho1:UniformPWP}-\ref{test case 2 :rho1:non-UniformPWP}, we employ the perturbed $L^{2}$ projection defined in \eqref{EQ:PL2:01}-\eqref{EQ:PL2:03} and stabilization parameter $\rho=1$. The numerical results demonstrate that $\|\nabla_d{e_{b}}\|_0$ converges in the superconvergence order of ${\cal{O}}(h^2)$, which perfectly consists with Corollary \ref{Corollary:SuperC-nonhomo}.
In Tables \ref{test case 2 :rho1:UniformWP} - \ref{test case 2 :rho1:non-UniformWP},  we use the usual $L^{2}$ projection  (i.e., the perturbation term $\varepsilon_b=0$) and stabilization parameter $\rho=1$. Table \ref{test case 2 :rho1:UniformWP} shows the convergence rate for $\| \nabla_d e_b\|_0$ is in a superconvergence order of ${\cal{O}}(h^2)$, which is in great consistency with   Corollary \ref{newsuperconvergence-H-1-10-10}, since the exact solution satisfies $u_{xx}=u_{yy}=u_{zz}$, the coefficient matrix $A$ is an identity matrix, and the partitions is uniform. Table \ref{test case 2 :rho1:non-UniformWP} shows that the superconvergence order in $H^{1}$- norm for $e_{b}$ seems to be $r=1.8$, which is higher than our theory $r=1.5$ in Theorem \ref{superconvergence-H-1-53}.  Tables \ref{test case 2 :rho1:UniformPWP}-\ref{test case 2 :rho1:non-UniformWP} show that a perturbed $L^{2}$ projection does provide a better numerical solution than the usual $L^{2}$ projection.

 Tables \ref{test case 2 :rho6:UniformPWP} - \ref{Example2:rho1:SquareWP:test4-111} show the numerical results for  the stabilization parameter $\rho=6$. Note that the perturbed $L^{2}$ projection defined in \eqref{EQ:PL2:01}-\eqref{EQ:PL2:03} turns to be the usual $L^{2}$ projection for the stabilization parameter $\rho=6$ on the uniform partitions. Table \ref{test case 2 :rho6:UniformPWP} shows the convergence rate for $\|\nabla_d{e_{b}}\|_0$ for the usual $L^{2}$ projection on the uniform partitions is in the supconvergence order of ${\cal{O}}(h^2)$ which is consistent with Corollary \ref{Corollary:SuperC-nonhomo}. In Table \ref{test case 2 :rho6:non-UniformPWP},  the perturbed $L^2$ projection on the non-uniform partitions is used and the numerical results show the superconvergence order for  $\|\nabla_d{e_{b}}\|_0$ is ${\cal{O}}(h^2)$, which consists with Corollary \ref{Corollary:SuperC-nonhomo}. In Table \ref{Example2:rho1:SquareWP:test4-111}, the usual  $L^2$ projection on non-uniform partitions is employed and it seems that  the convergence rate for  $\|\nabla_d{e_{b}}\|_0$ is in the order of ${\cal{O}}(h^{1.95})$, which is higher than  ${\cal{O}}(h^{1.5})$ in Theorem \ref{superconvergence-H-1-53}.

%
%
%
%

\begin{table}[htbp]\centering\scriptsize
\caption{Test Case 2: Convergence of the lowest order WG-FEM on the unit cubic domain with exact solution $u=\cos( x)\sin( y)\cos( z)$, uniform cubic partitions, stabilization parameter $\rho=1$, and perturbed $L^2$ projection of the Dirichlet boundary data $g$ by \eqref{EQ:PL2:01}-\eqref{EQ:PL2:03}.}\label{test case 2 :rho1:UniformPWP}
{
\setlength{\extrarowheight}{1.5pt}
\begin{center}
\begin{tabular}{|l|l|l|l|l|l|}
\hline $meshes$ &$\|u-\S(u_{b})\|_{\infty,\star}$& $\|e_0\|_0$ & $\|\nabla_d{e_{b}}\|_0$  &$\|\nabla_d u_b-\nabla u\|_{1,\star}$  &    $\|e_0\|_{1,1}$ \\
\hline
$4\times4\times4$     &2.8491e-02    &1.6837e-02    &3.5512e-02               &3.0139e-02     &3.4796e-02\\     \hline
$8\times8\times8$     &7.7591e-03    &4.2216e-03     &8.9019e-03                &7.5565e-03    &8.7222e-03 \\      \hline
$16\times16\times16$   &2.0045e-03    &1.0576e-03     &2.2294e-03                &1.8930e-03   &2.1845e-03   \\       \hline
$32\times32\times32$   &5.0788e-04     &2.6457e-04    &5.5769e-04               &4.7357e-04     &5.4645e-04\\           \hline
\hline
Rate        &1.98         &2.00      &2.00                      & 2.00  &2.00 \\
\hline
\end{tabular}
\end{center}
}
\end{table}
%
\begin{table}[htbp]\centering\scriptsize
\caption{Test Case 2: Convergence of the lowest order WG-FEM on the unit cubic domain with exact solution $u=\cos( x)\sin( y)\cos( z)$, non-uniform cubic partitions, stabilization parameter $\rho=1$, $h=\max(|e_x|,|e_y|,|e_z|)$, and perturbed $L^2$ projection of the Dirichlet boundary data $g$ by \eqref{EQ:PL2:01}-\eqref{EQ:PL2:03}.}\label{test case 2 :rho1:non-UniformPWP}
{
\setlength{\extrarowheight}{1.5pt}
\begin{center}
\begin{tabular}{|l|l|l|l|l|l|}
\hline $meshes$ &$\|u-\S(u_{b})\|_{\infty,\star}$& $\|e_0\|_0$ & $\|\nabla_d{e_{b}}\|_0$  &$\|\nabla_d u_b-\nabla u\|_{1,\star}$  &    $\|e_0\|_{1,1}$ \\
\hline
$3\times4\times5$     &4.5196e-02    &2.5473e-02    &5.1833e-02               &4.5731e-02     &5.1101e-02\\     \hline
$6\times8\times10$     &1.2338e-02    &6.3927e-03     &1.2987e-02                &1.1459e-02    &1.2803e-02 \\      \hline
$12\times16\times20$   &3.1901e-03    &1.6029e-03     &3.2538e-03                &2.8716e-03   &3.2078e-03   \\       \hline
$24\times32\times40$   &8.0847e-04    &4.0115e-04     &8.1411e-04               &7.1856e-04    &8.0261e-04\\           \hline
\hline
Rate                    &1.98          &2.00      &2.00                      &2.00         &2.00 \\
\hline
\end{tabular}
\end{center}
}
\end{table}

\begin{table}[htbp]\centering\scriptsize
\caption{Test Case 2: Convergence of the lowest order WG-FEM on the unit cubic domain with exact solution $u=\cos( x)\sin( y)\cos( z)$, uniform cubic partitions, stabilization parameter $\rho=1$, and $L^2$ projection of the Dirichlet boundary data $g$.}\label{test case 2 :rho1:UniformWP}
{
\setlength{\extrarowheight}{1.5pt}
\begin{center}
\begin{tabular}{|l|l|l|l|l|l|}
\hline $meshes$ &$\|u-\S(u_{b})\|_{\infty,\star}$& $\|e_0\|_0$ & $\|\nabla_d{e_{b}}\|_0$  &$\|\nabla_d u_b-\nabla u\|_{1,\star}$  &    $\|e_0\|_{1,1}$ \\
\hline
$4\times4\times4$     &9.6021e-03    &1.7217e-03    &1.6445e-03               &5.3190e-03     &1.6266e-03\\     \hline
$8\times8\times8$     &2.5944e-03    &4.3709e-04     &4.0413e-04                &1.3353e-03    &4.0262e-04 \\      \hline
$16\times16\times16$   &6.6871e-04    &1.1006e-04     &1.0087e-04                &3.3482e-04   &1.0093e-04   \\       \hline
$32\times32\times32$   &1.6933e-04    &2.7576e-05     &2.5230e-05               &8.3791e-05    &2.5280e-05\\           \hline
\hline
Rate        &1.98          &2.00      &2.00                      & 2.00  &2.00 \\
\hline
\end{tabular}
\end{center}
}
\end{table}
\begin{table}[htbp]\centering\scriptsize
\caption{Test Case 2: Convergence of the lowest order WG-FEM on the unit cubic domain with exact solution $u=\cos(x)\sin(y)\cos(z)$, non-uniform cubic partitions, stabilization parameter $\rho=1$, $h=\max(|e_x|,|e_y|,|e_z|)$, and $L^2$ projection of the Dirichlet boundary data $g$.}\label{test case 2 :rho1:non-UniformWP}
{
\setlength{\extrarowheight}{1.5pt}
\begin{center}
\begin{tabular}{|l|l|l|l|l|l|}
\hline $meshes$ &$\|u-\S(u_{b})\|_{\infty,\star}$& $\|e_0\|_0$ & $\|\nabla_d{e_{b}}\|_0$  &$\|\nabla_d u_b-\nabla u\|_{1,\star}$  &    $\|e_0\|_{1,1}$ \\
\hline
$3\times4\times5$     &1.0491e-02    &1.7126e-03    &9.4802e-03               &1.1082e-02     &9.5562e-03\\     \hline
$6\times8\times10$     &3.1746e-03    &4.4825e-04     &3.4760e-03                &3.7710e-03    &3.4972e-03 \\      \hline
$12\times16\times20$   &8.7237e-04    &1.1627e-04    &1.0666e-03                &1.1282e-03     &1.0715e-03  \\       \hline
$24\times32\times40$    &2.3618e-04    &2.9556e-05   &3.0668e-04             &3.2021e-04        &3.0782e-04\\           \hline
\hline
Rate                     &1.89          &1.98      &1.80                      & 1.82              &1.80 \\
\hline
\end{tabular}
\end{center}
}
\end{table}
%

\begin{table}[htbp]\centering\scriptsize
\caption{Test Case 2: Convergence of the lowest order WG-FEM on the unit cubic domain with exact solution $u=\cos( x)\sin( y)\cos( z)$, uniform cubic partitions, stabilization parameter $\rho=6$, and $L^2$ projection of the Dirichlet boundary data $g$.}
\label{test case 2 :rho6:UniformPWP}
{\setlength{\extrarowheight}{1.5pt}
\begin{center}
\begin{tabular}{|l|l|l|l|l|l|}
\hline $meshes$ &$\|u-\S(u_{b})\|_{\infty,\star}$& $\|e_0\|_0$ & $\|\nabla_d{e_{b}}\|_0$  &$\|\nabla_d u_b-\nabla u\|_{1,\star}$  &    $\|e_0\|_{1,1}$ \\
\hline
$4\times4\times4$     &9.6682e-03    &1.7514e-03    &1.5820e-03               &5.3468e-03     &1.5844e-03\\     \hline
$8\times8\times8$     &2.6003e-03    &4.4054e-04     &4.0144e-04                &1.3398e-03    &4.0227e-04 \\      \hline
$16\times16\times16$   &6.6911e-04    &1.1033e-04     &1.0080e-04                &3.3520e-04   &1.0104e-04   \\       \hline
$32\times32\times32$   &1.6935e-04     &2.7594e-05   &2.5229e-05                 &8.3818e-05    &2.5291e-05\\           \hline
\hline
Rate        &1.98          &2.00      &2.00 & 2.00  &2.00 \\
\hline
\end{tabular}
\end{center}
}
\end{table}
\begin{table}[htbp]\centering\scriptsize
\caption{Test Case 2: Convergence of the lowest order WG-FEM on the unit cubic domain with exact solution $u=\cos( x)\sin( y)\cos( z)$, non-uniform cubic partitions, stabilization parameter $\rho=6$, $h=(|e_x|^2+|e_y|^2+|e_z|^2)^{\frac{1}{2}}$, and perturbed $L^2$ projection of the Dirichlet boundary data $g$ by \eqref{EQ:PL2:01}-\eqref{EQ:PL2:03}.}\label{test case 2 :rho6:non-UniformPWP}
{
\setlength{\extrarowheight}{1.5pt}
\begin{center}
\begin{tabular}{|l|l|l|l|l|l|}
\hline $meshes$ &$\|u-\S(u_{b})\|_{\infty,\star}$& $\|e_0\|_0$ & $\|\nabla_d{e_{b}}\|_0$  &$\|\nabla_d u_b-\nabla u\|_{1,\star}$  &    $\|e_0\|_{1,1}$ \\
\hline
$3\times4\times5$     &4.5849e-03    &1.1666e-03    &5.8119e-03               &7.9357e-04     &5.0500e-03\\     \hline
$6\times8\times10$     &1.2310e-03    &2.8277e-04     &1.4518e-03                &1.9305e-04    &1.2608e-03 \\      \hline
$12\times16\times20$   &3.1661e-04    &7.0116e-05     &3.6286e-04                &4.8178e-05   &3.1510e-04   \\       \hline
$24\times32\times40$   &8.0124e-05    &1.7493e-05     &9.0708e-05               &1.2045e-05    &7.8768e-05\\           \hline
\hline
Rate        &1.98         &2.00      &2.00                      & 2.00  &2.00 \\
\hline
\end{tabular}
\end{center}
}
\end{table}
\begin{table}[htbp]\centering\scriptsize
\caption{Test Case 2: Convergence of the lowest order WG-FEM on the unit cubic domain with exact solution $u=\cos(x)\sin(y)\cos(z)$, non-uniform cubic partitions, stabilization parameter $\rho=6$, $h=(|e_x|^2+|e_y|^2+|e_z|^2)^{\frac{1}{2}}$, and $L^2$ projection of the Dirichlet boundary data $g$.}\label{Example2:rho1:SquareWP:test4-111}
{
\setlength{\extrarowheight}{1.5pt}
\begin{center}
\begin{tabular}{|l|l|l|l|l|l|}
\hline $meshes$ &$\|u-\S(u_{b})\|_{\infty,\star}$& $\|e_0\|_0$ & $\|\nabla_d{e_{b}}\|_0$  &$\|\nabla_d u_b-\nabla u\|_{1,\star}$  &    $\|e_0\|_{1,1}$ \\
\hline
$3\times4\times5$     &1.0575e-02    &1.8212e-03    &2.0566e-03               &6.1443e-03     &1.9932e-03\\     \hline
$6\times8\times10$     &2.8577e-03    &4.5857e-04     &5.2697e-04                &1.5434e-03    &5.1248e-04 \\      \hline
$12\times16\times20$   &7.3831e-04    &1.1496e-04     &1.3610e-04                &3.8768e-04   &1.3270e-04   \\       \hline
$24\times32\times40$   &1.8753e-04    &2.8764e-05     &3.5157e-05 & 9.7348e-05   &3.4342e-05\\           \hline
\hline
Rate        &1.98          &2.00      &1.95  & 1.99  &1.95 \\
\hline
\end{tabular}
\end{center}
}
\end{table}

{\bf Test Case 3 (Nonhomogeneous BVP)} In this group of numerical tests, the coefficient tensor $A$ is an identity matrix, the stabilization parameter is $\rho=1$, and the exact solution is $u=\cos(\pi x)\cos(\pi y)exp(z)$.  This is a nonhomogeneous boundary value problem.

Tables \ref{test case 3 :rho1:UniformWP}-\ref{test case 3 :rho1:UniformSWP} compare the performance on uniform cubic partitions when the perturbed $L^2$ projection and the usual $L^2$ projection are used, respectively.
 Table \ref{test case 3 :rho1:UniformWP} presents the convergence order for $\|\nabla_d e_b\|_0$ is in the superconvergence order of ${\cal{O}}(h^{1.7})$  for the usual $L^2$ projection, which is better than the theory ${\cal{O}}(h^{1.5})$ in Theorem \ref{superconvergence-H-1-53}. Table \ref{test case 3 :rho1:UniformSWP} demonstrates that the superconvergence rate for $\|\nabla_d{e_{b}}\|$ is ${\cal{O}}(h^2)$ for the perturbed $L^2$ projection, which consists perfectly with Corollary \ref{Corollary:SuperC-nonhomo}.

\begin{table}[htbp]\centering\scriptsize
\caption{Test Case 3: Convergence of the lowest order WG-FEM on the unit cubic domain with exact solution $u=\cos(\pi x)\cos(\pi y)exp(z)$, uniform cubic partitions, stabilization parameter $\rho=1$, and $L^2$ projection of the Dirichlet boundary data  $g$.}\label{test case 3 :rho1:UniformWP}
{
\setlength{\extrarowheight}{1.5pt}
\begin{center}
\begin{tabular}{|l|l|l|l|l|l|}
\hline $meshes$ &$\|u-\S(u_{b})\|_{\infty,\star}$& $\|e_0\|_0$ & $\|\nabla_d{e_{b}}\|_0$  &$\|\nabla_d u_b-\nabla u\|_{1,\star}$  &    $\|e_0\|_{1,1}$ \\
\hline
$4\times4\times4$     &1.3006e-01    &4.2491e-01    &4.8231e-01               &5.4528e-01     &4.8728e-01\\     \hline
$8\times8\times8$     &5.6002e-02    &1.3135e-02     &1.9955e-01                &2.1442e-01    &2.0121e-01 \\      \hline
$16\times16\times16$   &2.4039e-02    &3.7842e-03     &6.8308e-02                &7.1643e-02   &6.8723e-02   \\       \hline
$32\times32\times32$  &8.1938e-02     &1.0119e-03     &2.0898e-02                &2.1636e-02    &2.0993e-02\\           \hline
\hline
Rate    &1.55   &1.90  &1.71 & 1.73  &1.71 \\
\hline
\end{tabular}
\end{center}
}
\end{table}

\begin{table}[htbp]\centering\scriptsize
\caption{Test Case 3: Convergence of the lowest order WG-FEM on the unit cubic domain with exact solution $u=\cos(\pi x)\cos(\pi y)exp(z)$, uniform cubic partitions, stabilization parameter $\rho=1$, and perturbed $L^2$ projection of the Dirichlet boundary data $g$ by \eqref{EQ:PL2:01}-\eqref{EQ:PL2:03}.}\label{test case 3 :rho1:UniformSWP}
{
\setlength{\extrarowheight}{1.5pt}
\begin{center}
\begin{tabular}{|l|l|l|l|l|l|}
\hline $meshes$ &$\|u-\S(u_{b})\|_{\infty,\star}$& $\|e_0\|_0$ & $\|\nabla_d{e_{b}}\|_0$  &$\|\nabla_d u_b-\nabla u\|_{1,\star}$  &    $\|e_0\|_{1,1}$ \\
\hline
$4\times4\times4$     &4.5616e-01    &2.4448e-01    &9.6009e-01               &7.8292e-01     &9.2772e-01\\     \hline
$8\times8\times8$     &1.4466e-01    &6.0824e-02     &2.4368e-01                &1.9844e-01    &2.3531e-01 \\      \hline
$16\times16\times16$   &3.9370e-02    &1.5268e-02     &6.1494e-02                &5.0097e-02   &5.9376e-02   \\       \hline
$32\times32\times32$  &1.0149e-02      &3.8254e-03    &1.5432e-02                &1.2577e-02    &1.4901e-02\\           \hline
\hline
Rate                        &1.96          &2.00       &1.99                           & 1.99  &1.99 \\
\hline
\end{tabular}
\end{center}
}
\end{table}
Tables \ref{test case 3 :rho1:non-UniformPWP}-\ref{test case 3 :rho1:non-UniformWP} compare the performance   on non-uniform cubic partitions for the perturbed $L^2$ projection and the usual $L^2$ projection, respectively.  The convergence order shown in Table \ref{test case 3 :rho1:non-UniformPWP} is in good consistency with the theory  ${\cal{O}}(h^{2})$. The convergence order in Table \ref{test case 3 :rho1:non-UniformWP} seems to be in the order ${\cal{O}}(h^{1.7})$, which outforms the theory ${\cal{O}}(h^{1.5})$.

\begin{table}[htbp]\centering\scriptsize
\caption{Test Case 3: Convergence of the lowest order WG-FEM on the unit cubic domain with exact solution $u=\cos(\pi x)\cos(\pi y)exp(z)$, non-uniform cubic partitions, stabilization parameter $\rho=1$, $h=\max(|e_x|,|e_y|,|e_z|)$, and perturbed $L^2$ projection of the Dirichlet boundary data $g$ by \eqref{EQ:PL2:01}-\eqref{EQ:PL2:03}.}\label{test case 3 :rho1:non-UniformPWP}
{
\setlength{\extrarowheight}{1.5pt}
\begin{center}
\begin{tabular}{|l|l|l|l|l|l|}
\hline $meshes$ &$\|u-\S(u_{b})\|_{\infty,\star}$& $\|e_0\|_0$ & $\|\nabla_d{e_{b}}\|_0$  &$\|\nabla_d u_b-\nabla u\|_{1,\star}$  &    $\|e_0\|_{1,1}$ \\
\hline
$3\times4\times5$     &8.2655e-01    &4.4114e-01    &1.5682e+00               &1.3170e+00     &1.5251e+00\\     \hline
$6\times8\times10$     &2.7024e-01    &1.0965e-01     &3.9785e-01                &3.3358e-01    &3.8666e-01 \\      \hline
$12\times16\times20$   &7.3876e-02    &2.7530e-02     &1.0055e-01                &8.4341e-02   &9.7720e-02   \\       \hline
$24\times32\times40$   &1.9039e-02     &6.9048e-03    &2.5266e-02                  &2.1204e-02    &2.4556e-02\\           \hline
\hline
Rate                        &1.96         &2.00      &1.99                     & 1.99            &1.99 \\
\hline
\end{tabular}
\end{center}
}
\end{table}
\begin{table}[htbp]\centering\scriptsize
\caption{Test Case 3: Convergence of the lowest order WG-FEM on the unit cubic domain with exact solution $u=\cos(\pi x)\cos(\pi y)exp(z)$, non-uniform cubic partitions, stabilization parameter $\rho=1$, $h=\max(|e_x|,|e_y|,|e_z|)$, and $L^2$ projection of the Dirichlet boundary data $g$.}\label{test case 3 :rho1:non-UniformWP}
{
\setlength{\extrarowheight}{1.5pt}
\begin{center}
\begin{tabular}{|l|l|l|l|l|l|}
\hline $meshes$ &$\|u-\S(u_{b})\|_{\infty,\star}$& $\|e_0\|_0$ & $\|\nabla_d{e_{b}}\|_0$  &$\|\nabla_d u_b-\nabla u\|_{1,\star}$  &    $\|e_0\|_{1,1}$ \\
\hline
$3\times4\times5$     &1.5216e-01    &6.5348e-02    &6.6794e-01               &7.5389e-01     &6.7425e-01\\     \hline
$6\times8\times10$     &6.3091e-02    &2.3785e-02     &2.9490e-01                &3.1550e-01    &2.9745e-01 \\      \hline
$12\times16\times20$   &3.1023e-02   &6.9513e-03       &1.0408e-01             &1.0874e-01     &1.0475e-01\\       \hline
$24\times32\times40$   &1.0836e-02    &1.8540e-03      &3.2268e-02             &3.3300e-02     &3.2422e-02\\           \hline
\hline
Rate                     &1.52         &1.91             &1.69                          & 1.71       &1.69 \\
\hline
\end{tabular}
\end{center}
}
\end{table}

{\bf Test Case 4 (Nonhomogeneous BVP)}  In this test, the exact solution is $u=\sin(x)\sin(y)\sin(z)$, the diffusive tensor is $A=\bigl[\begin{smallmatrix}
10 & 3 & 1\\
3 & 2 & 1\\
1 & 1 & 2
\end{smallmatrix} \bigr]$, and the stabilization parameter is $\rho=1$. The numerical results shown in Table \ref{test case 4 :rho1:non-non-UniformCWP}  are based on the non-uniform cubic partitions and the usual $L^2$ projection. The numerical performance in Table \ref{test case 4 :rho1:non-non-UniformCWP} is in great consistency with Theorem \ref{superconvergence-H-1-53}.

\begin{table}[htbp]\centering\scriptsize
\caption{Test Case 4: Convergence of the lowest order WG-FEM on the unit cubic domain with exact solution $u=\sin( x)\sin( y)\sin(z)$, non-uniform cubic partitions, stabilization parameter $\rho=1$, $h=\max(|e_x|,|e_y|,|e_z|)$, and $L^2$ projection of the Dirichlet boundary data  $g$. The coefficient matrix is $a_{11}=10,a_{12}=3,a_{13}=1,$ and $a_{22}=2$,$a_{23}=1$,$a_{33
}=2$.}\label{test case 4 :rho1:non-non-UniformCWP}
{
\setlength{\extrarowheight}{1.5pt}
\begin{center}
\begin{tabular}{|l|l|l|l|l|l|}
\hline $meshes$ &$\|u-\S(u_{b})\|_{\infty,\star}$& $\|e_0\|_0$ & $\|\nabla_d{e_{b}}\|_0$  &$\|\nabla_d u_b-\nabla u\|_{1,\star}$  &    $\|e_0\|_{1,1}$ \\
\hline
$3\times4\times5$     &1.3451e-02    &7.2295e-03    &5.5726e-02               &5.5713e-02     &5.5722e-02\\     \hline
$6\times8\times10$     &8.6020e-03    &3.4567e-03     &2.9073e-02                &2.9050e-02    &2.9071e-02 \\      \hline
$12\times16\times20$   &4.2236e-03    &1.2471e-03     &1.1590e-02                &1.1583e-02   &1.1589e-02   \\       \hline
$24\times32\times40$  &1.5369e-03      &3.6709e-04    &3.9243e-03               &3.9227e-03     &3.9242e-03\\           \hline
\hline
Rate                   &1.46          &1.76     &1.56                     & 1.56  &1.56 \\
\hline
\end{tabular}
\end{center}
}
\end{table}

{\bf Test Case 5 (Nonhomogeneous BVP)} This test is in the following configuration: (1) The coefficient tensor $A$ is an identity matrix; (2) the stabilization parameter $\rho=1$; (3) the exact solution is $u=\cos(\pi x)\sin(\pi y)\cos(\pi z)$. The non-uniform cubic partition is generated by perturbing the uniform $N\times N\times N$ cubic partition with a random noise. More precisely, for any element $T=[x_{i},x_{i+1}]\times[y_{j},y_{j+1}]\times[z_{s},z_{s+1}]$ of the uniform $N\times N\times N$ cubic partition,  $x_{i+1}$, $y_{j+1}$ and $z_{s+1}$ are adjusted as follows: $
x_{i+1}^*=x_{i+1} + 0.2(\mbox{rand}(1)-0.5)/N$,
$y_{j+1}^*=y_{j+1} + 0.2(\mbox{rand}(1)-0.5)/N$, and
$z_{s+1}^*=z_{s+1} + 0.2(\mbox{rand}(1)-0.5)/N$,
where $\mbox{rand}(1)$ is the Matlab function which returns to a single uniformly distributed random number in $(0, 1)$.  The random numbers $\mbox{rand}(1)=\{0.141886,0.933993,0.031833\}$, $\mbox{rand}(1)=\{0.959492,0.392227,0.823457  \}$, and $\mbox{rand}(1)=\{0.421761,0.678735,0.276922   \}$ are used in the $x$-, $y$- and $z$- directions, respectively. Numerical results are presented in Tables \ref{test case 5 :rho1:non-non-Uniform} - \ref{test case 5 :rho1:non-non-UniformCWP} where the usual $L^2$ projection and the perturbed  $L^2$ projection are employed, respectively. Table \ref{test case 5 :rho1:non-non-Uniform}  demonstrates that the superconvergence order is much better than our theory for the usual $L^2$ projection.  Table \ref{test case 5 :rho1:non-non-UniformCWP}
shows that the superconvergence rate is in good consistency with the theory for the perturbed  $L^2$ projection.

\begin{table}[htbp]\centering\scriptsize
\caption{Test Case 5: Convergence of the lowest order WG-FEM on the unit cubic domain $(0,1)^{3}$ with exact solution $u=\cos(\pi x)\sin(\pi y)\cos(\pi z)$, non-uniform cubic partitions, the coefficient matrix is identity, stabilization parameter $\rho=1$, and $L^2$ projection of the Dirichlet boundary data $g$.}\label{test case 5 :rho1:non-non-Uniform}
{
\setlength{\extrarowheight}{1.5pt}
\begin{center}
\begin{tabular}{|l|l|l|l|l|l|}
\hline $meshes$ &$\|u-\S(u_{b})\|_{\infty,\star}$& $\|e_0\|_0$ & $\|\nabla_d{e_{b}}\|_0$  &$\|\nabla_d u_b-\nabla u\|_{1,\star}$  &    $\|e_0\|_{1,1}$ \\
\hline
$4\times4\times4$     &8.8874e-02    &1.0715e-02     &1.0106e-01                 &1.2318e-01     &9.0524e-02\\     \hline
$8\times8\times8$     &2.9802e-02    &2.6654e-03     &2.7306e-02                 &3.4129e-02     &2.5044e-02 \\      \hline
$16\times16\times16$  &8.3823e-03    &7.2443e-04     &7.7080e-03                 &9.5494e-03     &7.2704e-03   \\       \hline
$32\times32\times32$  &2.2344e-03    &1.8834e-04     &2.1549e-03                 &2.6033e-03     &2.0642e-03\\           \hline
\hline
Rate                   &1.91          &1.94     &1.84                     & 1.88  &1.82 \\
\hline
\end{tabular}
\end{center}
}
\end{table}

\begin{table}[htbp]\centering\scriptsize
\caption{Test Case 5: Convergence of the lowest order WG-FEM on the unit cubic domain $(0,1)^{3}$ with exact solution $u=\cos(\pi x)\sin(\pi y)\cos(\pi z)$, non-uniform cubic partitions, the coefficient matrix is identity, stabilization parameter $\rho=1$, and perturbed $L^2$ projection of the Dirichlet boundary data $g$ by \eqref{EQ:PL2:01}-\eqref{EQ:PL2:03}.}\label{test case 5 :rho1:non-non-UniformCWP}
{
\setlength{\extrarowheight}{1.5pt}
\begin{center}
\begin{tabular}{|l|l|l|l|l|l|}
\hline $meshes$ &$\|u-\S(u_{b})\|_{\infty,\star}$& $\|e_0\|_0$ & $\|\nabla_d{e_{b}}\|_0$  &$\|\nabla_d u_b-\nabla u\|_{1,\star}$  &    $\|e_0\|_{1,1}$ \\
\hline
$4\times4\times4$     &2.6832e-01    &1.5400e-01    &9.6036e-01                 &8.2564e-01     &9.4258e-01\\     \hline
$8\times8\times8$     &9.8770e-02    &3.9531e-02     &2.5404e-01                &2.1932e-01    &2.4937e-01 \\      \hline
$16\times16\times16$  &2.8979e-02    &1.0126e-02     &6.5344e-02                &5.6602e-02   &6.4164e-02   \\       \hline
$32\times32\times32$  &7.6431e-03    &2.5553e-03    &1.6504e-02                 &1.4316e-02     &1.6208e-02\\           \hline
\hline
Rate                   &1.92          &1.99     &1.99                     & 1.99  &1.99 \\
\hline
\end{tabular}
\end{center}
}
\end{table}

{\bf Test Case 6 (Nonhomogeneous BVP)} This test has the following configuration: (1) The coefficient tensor $A$ is an identity matrix; (2) the stabilization parameter is $\rho=1$; (3) the exact solution is $u=\cos(\pi x)\sin(\pi y)\cos(\pi z)$. The non-uniform cubic partitions is obtained in the same perturbation method as in Test Case 5. Table  \ref{Example6:rho1:cubicWP:testmesh5-2-6-14:14} shows that the convergence order for  $\|\nabla_d{e_{b}}\|_0$ with the usual $L^2$ projection is in the superconvergence order ${\cal{O}}(h^{1.8})$ which outperforms  the result ${\cal{O}}(h^{1.5})$  in Theorem \ref{superconvergence-H-1-53}.
Table \ref{Example6:rho1:cubicPWP:testmesh5-2-6-14:14} shows that the superconvergence order for  $\|\nabla_d{e_{b}}\|_0$  is in the order ${\cal{O}}(h^{2})$ with the perturbed $L^2$ projection, which is in great consistency with Corollary \ref{Corollary:SuperC-nonhomo}.

\begin{table}[htbp]\centering\scriptsize
\caption{Test Case 6: Convergence of the lowest order WG-FEM on the unit cubic domain $(0,1)^{3}$ with exact solution $u=\cos(\pi x)\sin(\pi y)\cos(\pi z)$, non-uniform cubic partitions, stabilization parameter $\rho=1$, and $L^2$ projection of the Dirichlet boundary data $g$.}\label{Example6:rho1:cubicWP:testmesh5-2-6-14:14}
{
\setlength{\extrarowheight}{1.5pt}
\begin{center}
\begin{tabular}{|l|l|l|l|l|l|}
\hline $meshes$ &$\|u-\S(u_{b})\|_{\infty,\star}$& $\|e_0\|_0$ & $\|\nabla_d{e_{b}}\|_0$  &$\|\nabla_d u_b-\nabla u\|_{1,\star}$  &    $\|e_0\|_{1,1}$ \\
\hline
$2\times2\times2$    &1.3717e-01    &4.8600e-02     &2.7327e-01               &4.3071e-01     &2.4263e-01\\     \hline
$4\times4\times4$    &8.4609e-02    &1.0223e-02     &9.0252e-02                &1.1798e-01    &8.0640e-02 \\      \hline
Rate              &0.70         & 2.25         &1.60                      & 1.87       &1.60 \\ \hline
\hline
$4\times4\times4$    &8.8874e-02    &1.0715e-02     &1.0106e-01               &1.2318e-01     &9.0524e-02\\     \hline
$8\times8\times8$    &2.9802e-02    &2.6654e-03     &2.7306e-02                &3.4129e-02    &2.5044e-02 \\      \hline
Rate              &1.58          & 2.01         &1.89                      & 1.85       &1.85 \\ \hline
\hline
$8\times8\times8$    &2.6820e-02    &2.4725e-03     &2.6568e-02               &3.3013e-02     &2.4314e-02\\     \hline
$16\times16\times16$    &7.5071e-03    &6.4926e-04     &7.1617e-03                &8.8805e-03    &6.6747e-03 \\      \hline
Rate              &1.84          & 1.93         &1.89                      & 1.89       &1.87 \\ \hline
\hline
$16\times16\times16$   &8.1166e-03    &7.4520e-04    &7.6548e-03     &9.5440e-03   &7.2359e-03\\     \hline
$32\times32\times32$   &2.1534e-03    &1.9459e-04    &2.2590e-03     &2.6909e-03   &2.1752e-03 \\      \hline
Rate        &1.91                & 1.94      &1.76                                  & 1.83  &1.73 \\   \hline
\hline
\end{tabular}
\end{center}
}
\end{table}
%
%
%
%
%
%
\begin{table}[htbp]\centering\scriptsize
\caption{Test Case 6: Convergence of the lowest order WG-FEM on the unit cubic domain $(0,1)^{3}$ with exact solution $u=\cos(\pi x)\sin(\pi y)\cos(\pi z)$, non-uniform cubic partitions, stabilization parameter $\rho=1$, and perturbed $L^2$ projection of the Dirichlet boundary data $g$ by \eqref{EQ:PL2:01}-\eqref{EQ:PL2:03}.}\label{Example6:rho1:cubicPWP:testmesh5-2-6-14:14}
{
\setlength{\extrarowheight}{1.5pt}
\begin{center}
\begin{tabular}{|l|l|l|l|l|l|}
\hline $meshes$ &$\|u-\S(u_{b})\|_{\infty,\star}$& $\|e_0\|_0$ & $\|\nabla_d{e_{b}}\|_0$  &$\|\nabla_d u_b-\nabla u\|_{1,\star}$  &    $\|e_0\|_{1,1}$ \\
\hline
$2\times2\times2$    &3.3996e-01    &6.0584e-01     &3.2721e+00               &2.7915e+00     &3.2143e+00\\     \hline
$4\times4\times4$    &2.4166e-01    &1.4017e-01     &8.7948e-01                &7.4759e-01    &8.6207e-01 \\      \hline
Rate              &0.49         & 2.11         &1.90                      & 1.90       &1.90 \\ \hline
\hline
$4\times4\times4$    &2.6832e-01    &1.5400e-01     &9.6036e-01               &8.2564e-01     &9.4258e-01\\     \hline
$8\times8\times8$    &9.8770e-02    &3.9531e-02     &2.5404e-01                &2.1932e-01    &2.4937e-01 \\      \hline
Rate              &1.44          & 1.96         &1.92                      & 1.91      &1.92 \\ \hline
\hline
$8\times8\times8$    &9.0683e-02       &3.7438e-02     &2.4135e-01               &2.0704e-01     &2.3677e-01\\     \hline
$16\times16\times16$    &2.6625e-02    &9.5967e-03     &6.2081e-02                &5.3436e-02    &6.0922e-02 \\      \hline
Rate              &1.77          & 1.96         &1.96                      & 1.95       &1.96 \\ \hline
\hline
$16\times16\times16$   &2.9695e-02    &1.0408e-02    &6.7150e-02     &5.8436e-02   &6.5996e-02\\     \hline
$32\times32\times32$   &7.8290e-03    &2.6270e-03    &1.6961e-02     &1.4781e-02   &1.6672e-02 \\      \hline
Rate        &1.92                & 1.99      &1.99                                  & 1.98  &1.98 \\   \hline
\hline
\end{tabular}
\end{center}
}
\end{table}
%
%
\subsection{Numerical experiments for piecewise constant diffusion tensor $A$}\label{subSection:20:29}

{\bf Test Case 7 (Nonhomogeneous BVP)} The domain $\Omega=(0,1)^{3}$ is divided into two subdomains by a flat face $x=1/2$, where $\Omega_1=(0,1/2)*(0,1)*(0,1)$ and $\Omega_2=(1/2,1)*(0,1)*(0,1)$. The diffusive coefficient tensor is $A_i=\bigl[\begin{smallmatrix}
\alpha_{i}^{x},0,0 \\
0,\alpha_{i}^{y},0\\
0,0,\alpha_{i}^{z}
\end{smallmatrix} \bigr]$, and the exact solution is $u_i=\alpha_{i}\cos(\pi x)\sin(\pi y)\cos(\pi z)$ for the subdomain $\O_i$, where the coefficients $\alpha^{x}_i,\alpha^{y}_i,\alpha^{z}_i,$ $\alpha_{i}$ are specified in Table \ref{Example8888:coefficient parameter}  for $i=1, 2$. The stablization parameter is $\rho=1$. Table \ref{Example7:rho1:CquWP} presents that the convergence rate for  $\|\nabla_d{e_{b}}\|_0$ is in the superconvergence order ${\cal{O}}(h^{1.8})$ on the non-uniform partitions with the usual $L^2$ projection, which is better than the theory ${\cal{O}}(h^{1.5})$.

\begin{table}[htbp]\centering\scriptsize
{\color{black}{\caption{Test Case 7: Parameters for the diffusive coefficients and the exact solution.} \label{Example8888:coefficient parameter}
{
\setlength{\extrarowheight}{1.5pt}
\begin{center}
\begin{tabular}{|l|l|}
\hline
$\alpha_{1}^{x}=1000$   &$\alpha_{2}^{x}=1$\\
$\alpha_{1}^{y}=100$    &$\alpha_{2}^{y}=0.1$\\
$\alpha_{1}^{z}=10$    &$\alpha_{2}^{z}=0.01$\\
$\alpha_{1}=0.01$    &$\alpha_{2}=10$    \\ \hline
\end{tabular}
\end{center}
}}
}
\end{table}
\begin{table}[htbp]\centering\scriptsize
{\color{black}{\caption{Test Case 7: Convergence of the lowest order WG-FEM on $(0,1)^{3}$ with exact solution $u=\alpha_{i}\cos(\pi x)\sin(\pi y)\cos(\pi z)$, piecewise constant diffusive tensor, non-uniform cubic partitions, stabilization parameter $\rho=1$, $h=\max(|e_x|,|e_y|,|e_z|)$, and $L^2$ projection of the boundary data $g$.} \label{Example7:rho1:CquWP}
{
\setlength{\extrarowheight}{1.5pt}
\begin{center}
\begin{tabular}{|l|l|l|l|l|l|}
\hline $meshes$ &$\|u-\S(u_{b})\|_{\infty,\star}$& $\|e_0\|_0$ & $\|\nabla_d{e_{b}}\|_0$  &$\|\nabla_d u_b-\nabla u\|_{1,\star}$  &    $\|e_0\|_{1,1}$ \\
\hline
$3\times4\times5$     &5.9145e-01 &5.1023e-01    &6.8587e+00               &1.1085e+01     &6.7852e+00\\     \hline
$6\times8\times10$     &5.6868e-01    &2.0684e-01     &3.4145e+00                &3.3386e+00    &3.4065e+00 \\      \hline
$12\times16\times20$   &1.5697e-01    &4.9439e-02     &9.9702e-01                &9.8059e-01   &9.9530e-01   \\       \hline
$24\times32\times40$ &3.9016e-02      &1.3197e-02     &2.8390e-01             &2.8028e-01    &2.8352e-01 \\           \hline
\hline
Rate      &2.01        &1.91         &1.81                        & 1.81     & 1.81 \\
\hline
\end{tabular}
\end{center}
}}
}
\end{table}
\subsection{Numerical experiments for variable diffusive tensor $A$}\label{subSection:NE3}
{\bf Test Case 8 (Nonhomogeneous BVP)} We consider a nonhomogeneous boundary value problem with the exact solution $u=\sin(x)\sin(y)\sin(z)$. The coefficient tensor $A$ is a symmetric and positive definite matrix with $a_{11}=1+x^{2}$, $a_{12}=xy/4$, $a_{13}=xz/4$, $a_{22}=1+y^{2}$, $a_{23}=yz/4$, $a_{33}=1+z^{2}$. The non-uniform cubic partitions and the usual $L^2$ projection are used in this test with the stabilized parameter $\rho=1$. Table \ref{test case 8-2 :rho1:non-UniformVWP-2} shows that the convergence rate for $\|\nabla_d{e_{b}}\|_0$ is of order ${\cal{O}}(h^{1.9})$ which outperforms the result ${\cal{O}}(h^{1.5})$ in Theorem \ref{superconvergence-H-1-53}.

\begin{table}[htbp]\centering\scriptsize
\caption{Test Case 8
: Convergence of the lowest order WG-FEM on the unit cubic domain with exact solution $u=\sin( x)\sin( y)\sin(z)$, non-uniform cubic partitions, stabilization parameter $\rho=1$, $h=\max(|e_x|,|e_y|,|e_z|)$, and $L^2$ projection of the Dirichlet boundary data $g$. The coefficient matrix is $a_{11}=1+x^2$, $a_{12}=xy/4$, $a_{13}=xz/4$, $a_{22}=1+y^2$, $a_{23}=yz/4$, and $a_{33
}=1+z^2$.}\label{test case 8-2 :rho1:non-UniformVWP-2}
{
\setlength{\extrarowheight}{1.5pt}
\begin{center}
\begin{tabular}{|l|l|l|l|l|l|}
\hline $meshes$ &$\|u-\S(u_{b})\|_{\infty,\star}$& $\|e_0\|_0$ & $\|\nabla_d{e_{b}}\|_0$  &$\|\nabla_d u_b-\nabla u\|_{1,\star}$  &    $\|e_0\|_{1,1}$ \\
\hline
$3\times4\times5$     &4.8515e-03    &9.3955e-04    &8.4062e-03               &8.8736e-03     &8.3704e-03\\     \hline
$6\times8\times10$     &1.3571e-03    &3.1310e-04     &3.7875e-03                &3.8168e-03    &3.7760e-03 \\      \hline
$12\times16\times20$   &3.9676e-04    &9.4737e-05     &1.1738e-03                &1.1752e-03   &1.1707e-03   \\       \hline
$24\times32\times40$  &1.1258e-04     &2.5404e-05     &3.1457e-04               &3.1448e-04     &3.1379e-04\\           \hline
\hline
Rate                    &1.82          &1.90            &1.90                     & 1.90      &1.90 \\
\hline
\end{tabular}
\end{center}
}
\end{table}

{\bf Test Case 9 (Reaction-diffusion equation)} Consider the reaction-diffusion model:
\begin{equation*}\label{re-D}
\begin{split}
-\Delta u+ cu =& f \quad \mbox{in}~~ \O=(0,1)^{3}, \\
u =& g\quad \mbox{on}~~ \pa\O,
\end{split}
\end{equation*}
where the reaction coefficient is $c=2$. The stabilizer parameter is $\rho=1$, and the exact solution is $u=x(1-x)y(1-2y)z(1-3z)$. The non-uniform cubic partitions and the usual $L^2$ projection are taken in the test. Table \ref{test case 7:rho1:Rect-di-WP} indicates that the convergence order for $\|\nabla_d{e_{b}}\|_0$ seems to be in a superconvergence order of ${\cal{O}}(h^{1.7})$. It should be pointed out that the reaction-diffusion model is not the second order elliptic model for which the superconvergence theory is established in the paper. However, the numerical results demonstrate a good computational performance of the WG finite element method for the reaction-diffusion model.

\begin{table}[htbp]\centering\scriptsize
\caption{Test Case 9: Convergence of the lowest order WG-FEM on the $(0,1)^{3}$ with exact solution $u=x(1-x)y(1-2y)z(1-3z)$, non-uniform cubic partitions, stabilization parameter $\rho=1$, $h=\max(|e_x|,|e_y|,|e_z|)$, and $L^2$ projection of the boundary data $g$. }\label{test case 7:rho1:Rect-di-WP}
{
\setlength{\extrarowheight}{1.5pt}
\begin{center}
\begin{tabular}{|l|l|l|l|l|l|}
\hline $meshes$ &$\|u-\S(u_{b})\|_{\infty,\star}$& $\|e_0\|_0$ & $\|\nabla_d{e_{b}}\|_0$  &$\|\nabla_d u_b-\nabla u\|_{1,\star}$  &    $\|e_0\|_{1,1}$ \\
\hline
$3\times4\times5$     &2.1502e-02    &7.5361e-03    &7.2672e-02               &7.4019e-02     &7.2672e-02\\     \hline
$6\times8\times10$     &1.1389e-02    &1.9955e-03     &2.6789e-02                &2.7184e-02    &2.6789e-02 \\      \hline
$12\times16\times20$   &3.7672e-03    &5.3715e-04     &8.6059e-03                &8.6998e-03   &8.6059e-03   \\       \hline
$24\times32\times40$  &1.0559e-03     &1.4113e-04     &2.6083e-03                &2.6290e-03    &2.6083e-03\\           \hline
\hline
Rate                    &1.84        &1.93            &1.72                     &1.73        &1.72 \\
\hline
\end{tabular}
\end{center}
}
\end{table}

In summary,  the superconvergence theory established in this paper is well verified by various numerical experiments. It is exciting that the convergence rate for $\|\nabla_d{e_{b}}\|_0$ is higher than the conclusion ${\cal{O}}(h^{1.5})$ in Theorem \ref{superconvergence-H-1-53} when the usual $L^2$ projection is taken. The numerical results show that the numerical solution related to the perturbed $L^2$ projection does perform better than the numerical solution related to the usual $L^2$ projection. Furthermore, the numerical solution for the reaction-diffusion equation shows a superconvergence error estimate by using the weak Galerkin scheme.

%
%
\begin{acknowledge}
The authors would like to gratefully acknowledge Dr. Junping Wang in NSF for his invaluable discussion and suggestion for this paper.
\end{acknowledge}


\begin{thebibliography}{99}


\bibitem{PEE_MAJT2000}  {\sc M. Ainsworth and J. T. Oden}, {\em
A posteriori error estimation in finite element analysis}, Wiley Interscience, New York, 2000.

\bibitem{HFS_AMC2001}   {\sc J. H. Brandts and M. K\v{R}\'{I}\v{Z}EK}, {\em History and future of superconvergence in three-dimensional finite element methods}, in proceedings of the conference on finite element methods: three-dimensional problems, GAKUTO Internat. Ser. Math. Sci. Appl. 15, Gakk\={o}tosho, Tokyo, pp. 24-35,  2001.

\bibitem{CBPE_MMA1996} {\sc
I. Babu\v{s}ka, T. Strouboulis, C. S. Upadhyay and S. K. Gangaray,}, {\em Computer-based proof of existence of superconvergence points in the finite element method; superconvergence of derivatives in finite element
solutions of Laplaces's, Poisson's and the elasticity equations}, Numer. Methods Partial Differential Equations., vol. 12, pp. 347-392, 1996.





\bibitem{HCJW-2003}  {\sc H. Chen and J. Wang}, {\em
An interior estimate of superconvergence for finite element solutions for second-order elliptic problems on quasi-uniform meshes by local projections},  SIAM J. Numer. Anal., vol. 41 (4) , pp. 1318-1338, 2003.

\bibitem{CMCYQH1995}  {\sc C. M. Chen and Y. Q. Huang}, {\em
High accuracy theory of finite elements (in Chinese)}, Hunan Science Press, Changsha, China. 1995.



\bibitem{SDG_SJAM2015}
{\sc W. Cao, C. Shu, Y. Yang and Z. Zhang}, {\em Superconvergence of discontinuous Galerkin methods for two-dimensional hyperbolic equations},
SIAM J. Numer. Anal., vol.  53(4), pp. 1651-1671, 2015.




\bibitem{REELJW-2002}  {\sc R. E. Ewing, M. Liu and J. Wang}, {\em
A new superconvergence for mixed finite element approximations},  SIAM J. Numer. Anal., vol. 40 (6) , pp. 2133-2150, 2002.


\bibitem{FLLWZ-2006}  {\sc G. Fairweather, Q. Lin, Y. Lin, J. Wang and S. Zhang}, {\em
Asymptotic expansions and richardson extrapolation of approximate solutions for second order elliptic problems on rectangular domains by mixed finite element methods},  SIAM J. Numer. Anal., vol. 44 (3) , pp. 1122-1149, 2006.



\bibitem{JAFRDG-2006}   {\sc J. A. Ferreira and R. D. Grigorieff}, {\em Supraconvergence and supercloseness of a scheme for elliptic equations on nonuniform grids}, Numer. Funct. Anal. Optim., vol. 27, pp. 539-564,  2006.

\bibitem{REXFZC-2018}   {\sc R. He, X. Feng and Z. Chen}, {\em $H^{1}$-superconvergence of a difference finite element method based on the $P_{1}-P_{1}$-conforming element on non-uniform meshes for the $3D$ Possion equation}, Math. Comp., vol. 87 (312), pp. 1659-1688,  2018.

\bibitem{AHSH2014}  {\sc A. Harris and S. Harris}, {\em
Superconvergence of weak Galerkin finite element approximation for second order elliptic problems by L2-projections},  Appl. Math. Comput., vol.  227, pp. 610-621, 2014.

\bibitem{AHSKMK2010}  {\sc A. Hannukainen, S. Korotov and M. Krizek}, {\em
Nodal $\OO(h^{4})$-superconvergence in $3D$ by averaging piecewise linear, bilinear, and trilinear FE approximations},  J. Comput. Math., vol.  28 (1), pp. 1-10, 2010.


\bibitem{HTJW-2002}  {\sc B. Heimsund, X. Tai and J. Wang}, {\em
Superconvergence for the gradient of finite element approximations by L2-projections},  SIAM J. Numer. Anal., vol. 40 (4) , pp. 1263-1280, 2002.


\bibitem{DWGWE_JSC2017}{\sc
Y. Huang, J. Li and D. Li}, {\em Developing weak Galerkin finite element method for the wave equation}, Numer. Meth. Partial Differential Equations., vol.  33(3), pp. 868-884, 2017.



\bibitem{MKPN1984}  {\sc M. Krizek and P. Neittaanm\"{a}ki}, {\em
Superconvergence phenomenon in the finite element method arising from averaging gradients},  Numer. Math., vol.  45, pp. 105-116, 1984.


\bibitem{VKPDL1986}  {\sc V. Kantchev and P. D. Lazarov}, {\em
Superconvergence of the gradient of linear finite elements for $3D$ possion equation},  Optimal Algorithms, Publ. Bulg. Acad. Sci., Sofia, pp. 172-182, 1986.

\bibitem{k2005}{\sc
M. Krizek}, {\em Superconvergence phenomenon on three-dimensional meshes}, International Journal of Numerical Analysis and Modeling, vol. 2(1), pp. 43-56, 2015.

\bibitem{kn1987}{\sc
M. Krizek and P. Neittaanmaki}, {\em On superconvergence techniques}, Acta Appl. Math., vol. 9, pp. 175-198, 1987.


\bibitem{LYZZ-WG-2018}  {\sc R. Lin, X. Ye, S. Zhang and P. Zhu}, {\em
A weak Galerkin finite element method for singularly perturbed by convection-diffusion-reaction problems},  SIAM J. Numer. Anal., vol. 56 (3) , pp. 1482-1497, 2018.



\bibitem{LJW2018}  {\sc Y. Liu and J. Wang}, {\em
A simiplified weak Galerkin finite element method: algorithm and error estimates}, https://arxiv.org/pdf/1808.08667v2.pdf.


\bibitem{RLZZ1008}  {\sc R. Lin and Z. Zhang}, {\em
Natural superconvergence points in three-dimensional finite elements},  SIAM J. Numer. Anal., vol.  46 , pp. 1281-1297, 2008.


\bibitem{sup_LWW2018}{\sc D. Li, C. Wang and J. Wang}, {\em
Superconvergence of the gradient approximation for weak Galerkin finite element methods on nonuniform rectangular partitions}, https://arxiv.org/pdf/1804.03998v2.pdf.

\bibitem{sup_LW2018}{\sc Y. Liu and J. Wang}, {\em
Simplified weak Galerkin and finite difference schemes for the stokes equation}, https://arxiv.org/pdf/1803.0012v1.pdf.

\bibitem{sup_CRC2016} {\sc G. R. Liu and T. Nguyen-Thoi},
{\em Smoothed finite element methods}. CRC press. 2016.


\bibitem{CSA_SIAM2016}  {\sc K. Mustapha, M. Nour and B. Cockburn}, {\em
Convergence and superconvergence analyses of HDG methods for time fractional diffusion problems},  Adv. Comput. Math.,  vol. 42 (2), pp. 377-393, 2016.



\bibitem{ANWG_EIP2016}
 {\sc L. Mu, J. Wang, X. Ye and S. Zhao}, {\em
 A new weak Galerkin finite element method for elliptic interface problems},
 J. Comput. Phy., vol. 325, pp. 157-173, 2016.



\bibitem{MWYS_MAX2015}
 {\sc L. Mu, J. Wang, X. Ye and S. Zhang}, {\em
 A weak Galerkin finite element method for the Maxwell equations},
 J. Sci. Comput., vol. 65, pp. 363-386, 2015.

 \bibitem{MWYS_HE2014}
 {\sc L. Mu, J. Wang, X. Ye and S. Zhao}, {\em
 Numerical studies on the weak Galerkin method for the Helmholtz equation with large wave number},
 Communications in Computations in Computational Physics., vol. 15, pp. 1461-1474, 2014.

\bibitem{ellip_NWA2013}  {\sc L. Mu, J. Wang, Y. Wang and X. Ye}, {\em
A computational study of the weak Galerkin method for second-order elliptic equations},   Numer.  Algor., vol. 63, pp. 753-777, 2013.



\bibitem{AHSIHLB-1996}  {\sc A. H. Schatz, I. H. Sloan and L. B. Wahlbin}, {\em Superconvergence in finite element methods and meshes that are locally symmetric with respect to a point}, SIAM J. Numer. Anal., vol.  33, pp. 505-521, 1996.

\bibitem{WG_ME2017}  {\sc S. Shields, J. Li and E. A. Machorro}, {\em
Weak Galerkin methods for time-dependent Maxwell's equations},  Comput. Math. Appl., vol.  74, pp. 2106-2124, 2017.

\bibitem{ellip_JW2000}  {\sc J. Wang}, {\em
A superconvergence analysis for finite element solutions by the least-squares surface fitting on irregular meshes for smooth problems}, J. Math. Study., vol. 33(3), pp. 229-243, 2000.

\bibitem{sup_GFEM2006} {\sc L. Wahlbin}, {\em  Superconvergence in Galerkin finite element methods}, Springer. 2006.



\bibitem{SAPP_JSC2018} {\sc R. Wang, R. Zhang, X. Zhang and Z. Zhang}, {\em
Superconvergence analysis and polynomial preserving Recovery for a class of weak Galerkin Methods}, Numer. Meth. Partial Differential Equations., vol. 34 (1), pp. 317-335, 2018.


\bibitem{PDWG_MC2018} {\sc C. Wang and J. Wang}, {\em
A primal-dual weak Galerkin finite element method for second order elliptic equations in non-divergence form}, Math. Comp., vol. 87, pp. 515-545, 2018.


\bibitem{WZZZ_MC2018} {\sc J. Wang, Q. Zhai, R. Zhang and S. Zhang}, {\em
A weak Galerkin finite element scheme for the Cahn-Hilliard equation}, Math. Comp., vol. 88 (315), pp. 211-235, 2018.



\bibitem{ellip_WY2013} {\sc J. Wang and X. Ye}, {\em
A weak Galerkin finite element method for second-order elliptic problems}, J. Comput. Appl. Math., vol. 307, pp. 103-115, 2013.


\bibitem{ellip_MC2014} {\sc J. Wang and X. Ye}, {\em
A weak Galerkin mixed finite element method for second-order elliptic problems}, Math. Comp., vol. 83, pp. 2101-2126, 2014.

\bibitem{wangsup} {\sc C. Wang}, {\em Superconvergence of Ritz-Galerkin finite element approximations for second order elliptic problems}, Numerical Methods for Partial Differential Equations., vol. 34, pp. 838-856, 2018.

\bibitem{fp2018} {\sc C. Wang and J. Wang}, {\em
A primal-dual weak Galerkin finite element method for Fokker-Planck type equations}, https://arxiv.org/pdf/1704.05606.pdf, SIAM Journal of Numerical Analysis, accepted.

\bibitem{ECP2018} {\sc C. Wang and J. Wang}, {\em
Primal-dual weak Galerkin finite element methods for elliptic cauchy problems}, https://arxiv.org/pdf/1806.01583.pdf.

\bibitem{CW2018}  {\sc C. Wang}, {\em
A new primal-dual weak Galerkin finite element method for ill-posed elliptic cauchy problems}, https://arxiv.org/pdf/1809.04697v1.pdf.


\bibitem{WWmaxwell}  {\sc C. Wang}, {\em New discretization schemes for time-harmonic Maxwell equations by weak Galerkin finite element methods}, Journal of Computational and Applied Mathematics, Vol. 341, pp. 127-143, 2018.



\bibitem{WWdivcurl}  {\sc C. Wang and J. Wang}, {\em Discretization of div-curl systems by weak Galerkin finite element methods on polyhedral partitions}, Journal of Scientific Computing, Vol.  68, pp. 1144-1171, 2016.


\bibitem{WWhybird}  {\sc C. Wang and J. Wang}, {\em A hybridized formulation for weak Galerkin finite element methods for biharmonic equation on polygonal or polyhedral meshes}, International Journal of Numerical Analysis and Modeling, Vol. 12, pp. 302-317, 2015.

\bibitem{WWscience}  {\sc J. Wang and C. Wang}, {\em Weak Galerkin finite element methods for elliptic PDEs}, Science China, Vol. 45, pp. 1061-1092, 2015.

\bibitem{WWbiharmonic}  {\sc C. Wang and J. Wang}, {\em An efficient numerical scheme for the biharmonic equation by weak Galerkin finite element methods on polygonal or polyhedral meshes}, Journal of Computers and Mathematics with Applications, Vol. 68, 12, pp. 2314-2330, 2014.


\bibitem{WCZ-2014}  {\sc H. Wei, L. Chen and B. Zheng}, {\em
Adaptive mesh refinement and superconvergence for two-dimensional interface problems},  SIAM J. Sci. Comput., vol. 36 (4) , pp. A1478-A1499, 2014.


\bibitem{JWYSE-2016}  {\sc J. Wang X. Ye}, {\em
A weak Galerkin finite element method for the stokes equations},  Adv. Comput. Math., vol. 42 , pp. 155-174, 2016.

 \bibitem{STOKES_NA2017} {\sc X. Zheng and X. Xie}, {\em
A posterior error estimator for a weak Gakerkin finite element solution of the stokes problem}, East Asian Journal on Applied Mathematics, vol. 7(3), pp. 508-529, 2017.

\bibitem{sup_CMAME1992} {\sc O. Zienkiewicz and J. Zhu}, {\em
The superconvergence patch recovery (SPR) and adaptive finite element refinement}, Comput. Methods Appl. Mech. Eng., vol. 101 (1-3), pp. 207-224, 1992.



\bibitem{supposter_PIJNM11992} {\sc O. Zienkiewics and J. Zhu}, {\em
The superconvergence patch recovery and a posteriori error estimates}, Part 1,   Internat. J. Numer. Methods Engrg., vol.  33, pp.  1331-1364, 1992.

\bibitem{supposter_PIJNM21992} {\sc O. Zienkiewics and J. Zhu}, {\em
The superconvergence patch recovery and a posteriori error estimates}, Part 2,  Internat. J. Numer. Methods Engrg., vol. 33, pp. 1365-1382, 1992.



\bibitem{sup_QDZQL1989} {\sc Q. D. Zhu and Q. Lin}, {\em
Superconvergence theory of the finite element method}, Hunan Science Press, China, Changsha, 1989.

\bibitem{MZ_MMA2002} {\sc M. Zlamal}, {\em
Some superconvergence results in the finite element method}. In A. Dold and B. Eckmann, editors, Mathematical Aspects of Finite Element Methods, number 606 in Springers Lecture Notes in Mathematics, 1975.


\bibitem{ZZ1998}  {\sc Z. Zhang}, {\em
Derivative superconvergence points in finite element solutions of poission's equation for the serendipity and intermediate families--a theoretical justification},  Math. Comp., vol.  67, pp. 541-552, 1998.











\end{thebibliography}
\end{document}